\documentclass[preprint]{imsart}
\usepackage{amsmath,amsthm,verbatim,amssymb,amsfonts,amscd, graphicx}
\usepackage{amssymb,bm}
\usepackage{mathrsfs}
\usepackage{subfig}
\RequirePackage[colorlinks,citecolor=blue,urlcolor=blue]{hyperref}
\RequirePackage[numbers]{natbib}
\usepackage{constants}
\usepackage{enumerate}

\newcommand{\rom}[1]{\uppercase\expandafter{\romannumeral #1\relax}}
\newcommand{\beas}{\begin{eqnarray*}}
\newcommand{\enas}{\end{eqnarray*}}
\newcommand{\bea}{\begin{eqnarray}}
\newcommand{\ena}{\end{eqnarray}}
\newcommand{\bms}{\begin{multline*}}
\newcommand{\ems}{\end{multline*}}
\newcommand{\bels}{\begin{align*}}
\newcommand{\enls}{\end{align*}}
\newcommand{\bel}{\begin{align}}
\newcommand{\enl}{\end{align}}

\newcommand{\ignore}[1]{}
\newcommand{\tr}{\mbox{tr\,}}

\startlocaldefs

\numberwithin{equation}{section}

\newtheorem{theorem}{Theorem}[section]

\newtheorem{corollary}{Corollary}[section]

\newtheorem{lemma}{Lemma}[section]
\newtheorem{assumption}{Assumption}
\newtheorem{remark}{Remark}[section]

\newtheorem{fact}{Fact}

\endlocaldefs

\def\blfootnote{\xdef\@thefnmark{}\@footnotetext}

\newcommand{\dotp}[2]{\left\langle#1,#2\right\rangle}
\newcommand{\m}{\mathcal}
\newcommand{\mb}{\mathbb}
\newcommand\argmin{\mathop{\mbox{argmin}}}

\newcommand{\card}{\mathrm{Card}}
\newcommand{\sign}{\mathrm{sign}}
\def\r{\right}
\def\l{\left}
\newcommand{\eps}{\varepsilon}
\newcommand{\var}{\mbox{Var}}

\newcommand{\wh}{\widehat}

\newcommand{\pr}[1]{\mathrm{Pr}{\left(#1 \right)}}

\begin{document}

\begin{frontmatter}
\title{Uniform bounds for robust mean estimators}
\runauthor{S. Minsker}
\runtitle{Robust ERM}

\begin{aug}
\author{\fnms{Stanislav} \snm{Minsker}\ead[label=e1,mark]{minsker@usc.edu} \thanksref{e1,a}}

\address[a]{Department of Mathematics, University of Southern California \\
\printead{e1}}

\thankstext{t3}{Supported in part by the National Science Foundation grant DMS-1712956.}
\end{aug}

\begin{abstract}
This paper is devoted to the estimators of the mean that provide strong non-asymptotic guarantees under minimal assumptions on the underlying distribution. 
The main ideas behind proposed techniques are based on bridging the notions of symmetry and robustness. 
We show that existing methods, such as median-of-means and Catoni's estimators, can often be viewed as special cases of our construction. 
The main contribution of the paper is the proof of uniform bounds for the deviations of the stochastic process defined by proposed estimators. Moreover, we extend our results to the case of adversarial contamination where a constant fraction of the observations is arbitrarily corrupted. Finally, we apply our methods to the problem of robust multivariate mean estimation and show that obtained inequalities achieve optimal dependence on the proportion of corrupted samples. 

\end{abstract}

\end{frontmatter}

\section{Introduction}
\label{sec:intro}

Let $(S,\m S)$ be the measurable space, and $X\in S$ be a random variable with distribution $P$. 
Moreover, suppose that $X_1,\ldots,X_N$ are i.i.d. copies of $X$. 
Assume that $\m F$ is a class of measurable functions from $S$ to $\mb R$. 
Many problems in mathematical statistics and statistical learning theory require simultaneously estimating $Pf:=\mb Ef(X)$ for all $f\in \m F$. For example, in the maximum likelihood estimation framework, $\m F=\l\{ \log p_\theta(\cdot), \ \theta\in \Theta \r\}$ is a family of probability density functions with respect to a $\sigma$-finite measure $\mu$, and $\frac{dP}{d\mu}=p_{\theta_\ast}$ for $\theta_\ast \in\Theta$. 
The most common way to estimate $\mb Ef(X)$ is via the empirical mean $P_N f:=\frac{1}{N}\sum_{j=1}^N f(X_j)$. 
Deviations of the resulting \emph{empirical process} $\m F\ni f\mapsto \sqrt{N}(P_N - P)f$ have been extensively studied, however, sharp estimates are known only under rather restrictive conditions, such as the case when functions in $\m F$ are uniformly bounded, or when the envelope $F(x) := \sup_{f\in \m F}|f(x)|$ of the class $\m F$ possesses finite exponential moments \cite{wellner1,kolt6,bartlett2005local,geer2000empirical,Adamczak:2008aa}.  

Here, we consider the situation when the random variables $\{f(X), \ f\in \m F\}$ indexed by $\m F$ are allowed to be heavy-tailed, meaning that they possess finite moments of low order only (in the context of this paper, ``low order'' will usually mean the range between $2$ and $3$). 
In this case, the tail probabilities $P\l( \l|\frac{1}{N}\sum_{j=1}^N f(X_j) - \mb Ef(X) \r| \geq t \r)$ decay polynomially, making many existing techniques unapplicable.  
Our approach to simultaneous mean estimation is based on replacing the sample mean by a different estimator of $\mb Ef(X)$ that is ``robust'' to heavy tails and admits tight concentration under minimal moment assumptions. 
Well known examples of such estimators include the median-of-means (MOM) estimator \cite{Nemirovski1983Problem-complex00,alon1996space,lerasle2011robust} and Catoni's estimator \cite{catoni2012challenging}. 
These two techniques rely on different principles for controlling the bias: while Catoni's estimator is (informally speaking) based on delicate ``truncation,'' the MOM estimator exploits the fact that both the median and the mean of a symmetric distribution coincide with the center of symmetry \cite{minsker2017distributed}.  
Construction proposed in this work shows that these principles can be unified. We suggest a family of estimators that can be viewed as a ``bridge'' between Catoni's estimator and the MOM technique, and prove uniform bounds for the deviations of the resulting stochastic process. 
We also address the more challenging framework of adversarial contamination and show that estimators of the mean of a random vector obtained using our methods admit optimal performance bounds in this case. 

\subsection{Notation and organization of the paper.}
\label{sec:definitions}

Absolute positive constants will be denoted $c,C,c_1$, etc., and may take different values in different parts of the paper. 
For a function $h:\mb R\mapsto\mb R$, define 
\[
\argmin_{z\in\mb R} h(z) := \{z\in\mb R: h(z)\leq h(x)\text{ for all }x\in \mb R\},
\]
and $\|h\|_\infty:=\mathrm{ess \,sup}\{ |h(x)|: \, x\in \mb R\}$. For a Lipschitz continuous function $h$, $L(h)$ will denote its Lipschitz constant. 
For $f\in \m F$, denote $\sigma^2(f) = \var(f(X))$ and $\sigma^2(\m F) = \sup_{f\in \m F}\sigma^2(f)$. 
Everywhere below, $\Phi(\cdot)$ stands for the cumulative distribution function of the standard normal random variable, and $W(f)$ denotes a random variable with distribution $N\l( 0,\sigma^2(f) \r)$. Additional notation and auxiliary results are introduced as necessary.

Material of the paper is organized as follows: section \ref{sec:construction} explains the main ideas behind construction of the estimators studied in the paper. The main results are  stated in section \ref{sec:M-est}, followed by discussion and comparison to the literature on the topic in section \ref{sec:comparison}. 
Section \ref{sec:adversary} discuss extensions of the results in the framework of adversarial contamination. 
Finally, section \ref{section:mean} explores implications of the bounds for the problem of multivariate mean estimation. 
Finally, the proofs are presented in section \ref{sec:proofs}.

\section{Construction of robust estimators of the mean.}
\label{sec:construction}

Proposed estimators are based on the following (informally stated) principles: 
\begin{enumerate}[(a)]
\item  If the distribution $Q$ is symmetric, then its center of symmetry $\theta(Q)$ can be approximated by a robust estimator, such as Huber's robust M-estimator of location \cite{huber1964robust} defined via
\begin{center}
$\widehat\theta := \argmin\limits_{z\in \mb R} \sum_{j=1}^N \rho\left( z - Y_j\right)$,
\end{center}
where $Y_1,\ldots,Y_N$ is an i.i.d. sample from $Q$ and $\rho$ is a convex, even function with bounded derivative.
\item In order to construct a robust estimator of a parameter $\theta^\diamond(Q)$ of (not necessarily symmetric) distribution $Q$ based on an i.i.d. sample $Y_1,\ldots,Y_N$, create an auxillary sample $\xi_1,\ldots,\xi_M$ such that
\begin{enumerate}
\item[(i)] it is governed by an approximately symmetric distribution;
\item[(ii)] the center of symmetry of this distribution is close to $\tilde\theta(P)$. 
\end{enumerate}
According to (a), we then define an estimator of $\theta^\diamond(Q)$ via 
\[
\widehat \theta^\diamond := \argmin\limits_{z\in \mb R} \sum_{j=1}^M \rho\left( z - \xi_j\right).
\] 

\end{enumerate}

The main focus of this work is the case when $Q$ is the distribution of $f(X)$ and $\theta^\diamond(Q)$ corresponds to the mean $\mb E f(X)$. 
To construct a ``new sample'' $\xi_1,\ldots,\xi_M$ governed by an approximately symmetric distribution centered at $\mb Ef(X)$, we rely on the fact that, under mild assumptions, the sample mean is asymptotically normal, hence asymptotically symmetric.   
Let $k$ be an integer, and assume that $G_1,\ldots,G_k$ are subsets of the index set $\{1,\ldots,N\}$ of cardinality $|G_j| = n:= \lfloor N/k\rfloor$ each, where the partition method is independent of the data $X_1,\ldots,X_N$. In general, we do not require the subsets to be disjoint, and different possibilities will be discussed in the following sections. 
Let 
\[
\bar \theta_j(f) = \frac{1}{n}\sum_{i\in G_j} f(X_i)
\]
be the empirical mean evaluated over the subsample indexed by $G_j$. 
Conditions on the suitable ``loss function'' $\rho$ are summarized in the following assumption. 
\begin{assumption}
\label{ass:1}
Suppose that $\rho: \mb R\mapsto \mb R$ is a convex, even, continuously differentiable function such that
\begin{enumerate}
\item[(i)] $\rho'(z)=z$ for $|z|\leq 1$ and $\rho'(z)=\mathrm{const}$ for $z\geq 2$.
\item[(ii)] $z - \rho'(z)$ is nondecreasing; 
\end{enumerate}
\end{assumption}
\noindent For instance, Huber's loss \cite{huber1964robust}
\begin{equation*}
\rho(z)=\frac{z^2}{2} I\{|z|\leq 1\} + \l( |z| - 1/2\r) I\{|z|>1\}
\end{equation*}
is an example of a function satisfying Assumption \ref{ass:1}.
\begin{remark}
\label{rem:sup}
Assumption \ref{ass:1} implies that $\rho'(2) - 2\leq \rho'(1)-1 = 0$, hence $\|\rho'\|_\infty \leq 2$. Moreover, for any $x>y$, 
$\rho'(x) - \rho'(y) = y - \rho'(y) - (x - \rho'(x)) + x - y \leq x-y$, hence $\rho'$ is Lipschitz continuous with Lipschitz constant $L(\rho')=1$. 
\end{remark}
Note that the loss $\rho(x)=|x|$ that leads to a ``classical'' median-of-means estimator does not satisfy assumption \ref{ass:1}, in particular, $\rho'$ is not Lipschitz continuous. Lipschitz continuity of $\rho'$ turns out to be crucial for the derivation of our uniform bounds, however, for finite classes $\m F$, the proofs are valid for broader class of losses that includes $\rho(x)=|x|$; see \cite{minsker2017distributed} for more details. 
The principles and assumptions stated above lead to the following definition of robust mean estimators: given $\Delta > 0$, set
\begin{align}
\label{eq:M-est}
\wh\theta^{(k)}(f):= \argmin_{z\in \mb R}\frac{1}{\sqrt{N}}\sum_{j=1}^k \rho\l(\sqrt{n}\,\frac{\bar \theta_j(f) - z}{\Delta}\r).
\end{align}
Parameter $\Delta$, when expressed on a ``natural scale'' of the problem defined by $\sigma(\m F)$, can be interpreted as the truncation level. It will be shown that different combinations of the subgroup size $n$ and ``truncation level'' $\Delta$ may lead to equally good bounds: in particular, when $n=1$ and $\Delta\propto \sigma(\m F)\sqrt{N}$, we will recover Catoni's estimator, while the case of large $n$ (e.g., $n\simeq \sqrt{N}$) and $\Delta\propto \sigma(\m F)$ leads to the MOM-type estimator. 

\subsection{Main results.}
\label{sec:M-est}

The collection of random variables $\l\{ \wh\theta^{(k)}(f) - Pf,  \ f\in \m F\r\}$ defines a stochastic process that is a natural analogue of the empirical process in the framework of heavy-tailed data. 
Our main goal is to characterize the size of the supremum of the process, namely
\footnote{We assume everywhere below that $\sup_{f\in \m F} \l| \wh\theta^{(k)}(f) - Pf \r|$ is properly measurable. See \cite{wellner1,dudley2014uniform} for in-depth discussion of measurability issues.}
\[
\sup_{f\in \m F} \l| \wh\theta^{(k)}(f) - Pf \r|.
\]
In particular, we will be interested in estimating the deviation probabilities $P\l( \sup_{f\in \m F} \l| \wh\theta^{(k)}(f) - Pf \r| \geq t\r)$ under minimal assumptions on the process $\l\{ f(X), \ f\in \m F\r\}$. 
As a corollary of our general bounds, we will obtain new results for the problem of mean estimation in $\mb R^d$. 
Everywhere below, it will be assumed that $\sigma(\m F)<\infty$.

As a first step, we introduce the main quantities appearing in our bounds. 
Given $f\in \m F$ such that $\sigma(f)>0$, $n\in \mb N$ and $t>0$, define
\begin{multline*}
g_f(t,n):= C \Bigg( \frac{ \mb E (f(X) - \mb Ef(X))^2 \, I\l\{ \frac{ |f(X)-\mb Ef(X)|}{\sigma(f)\sqrt{n}} >  1+\l|\frac{t}{\sigma(f)}\r| \r\} }{\sigma^2(f) \l(1+\l|\frac{t}{\sigma(f)}\r| \r)^2} 
\\
+ \frac{1}{\sqrt{n}}
\frac{\mb E |f(X)-\mb Ef(X)|^3 \, I\l\{ \frac{|f(X) - \mb Ef(X)|}{\sigma(f)\sqrt{n}} \leq 1+ \l|\frac{t}{\sigma(f)}\r|\r\}}{\sigma^3(f) \l(1+\l|\frac{t}{\sigma(f)}\r| \r)^3}\Bigg).
\end{multline*}
It follows from the results of L. Chen and Q.-M. Shao \cite[Theorem 2.2 in][]{chen2001non} that $g_f(t,n)$ controls the rate of convergence in the central limit theorem, namely 
\begin{equation}
\label{eq:BE-main}
\l| \pr{ \frac{\sum_{j=1}^n \l(f(X_j) - Pf\r)}{\sigma(f)\sqrt{n}} \leq t} - \Phi( t ) \r| \leq g_f(t,n),
\end{equation}
given that $\sigma^2(f)<\infty$ and that an absolute constant $C$ in the definition of $g_f(t,n)$ is large enough. 
The function $g_f(t,n)$ enters our bounds through the quantity that we define next. Given $\Delta>0$, set 
\[
G_f(n,\Delta) := \int_{0}^{\infty} g_f\l(\Delta\l( \frac{1}{2} + t \r),n \r)dt.
\] 
The following statement provides simple upper bounds for $g_f(t,n)$ and $G_f(n,\Delta)$. 
\begin{lemma}
\label{lemma:BE-nonunif}
Let $X_1,\ldots,X_n$ be i.i.d. copies of $X$, and assume that $\var(f(X))<\infty.$ Then $g_f(t,n)\to 0$ as $|t|\to\infty$ and $g_f(t,n)\to 0$ as $n\to\infty$, with convergence being monotone. Moreover, if $\mb E| f(X) - \mb Ef(X)|^{2+\delta}<\infty$ for some $\delta\in (0,1]$, then for all $t>0$
\begin{align}
\label{eq:g_f}
g_f(t,n)&\leq C'  \frac{\mb E\big|f(X) - \mb Ef(X) \big|^{2+\delta}}{n^{\delta/2}\l(\sigma(f)+ \l| t \r| \r)^{2+\delta}} \leq C' \frac{\mb E\big|f(X) - \mb Ef(X)\big|^{2+\delta}}{n^{\delta/2} |t|^{2+\delta}},
\\
\nonumber
G_f(n,\Delta) & \leq C'' \frac{\mb E\big|f(X) - \mb Ef(X)\big|^{2+\delta}}{\Delta^{2+\delta} n^{\delta/2}},
\end{align}
where $C',C''>0$ are absolute constants. 
\end{lemma}
\noindent The proof of this lemma is outlined in section \ref{proof:BE-nonunif}. 
We are now ready to state the main result. 
Let $\rho$ be a loss function satisfying Assumption \ref{ass:1}. 
Moreover, set 
\[
\widetilde \Delta:=\max\l( \Delta,\sigma(\m F)\r).
\]
\begin{theorem}
\label{th:unif}
Assume that $N=nk$ and that the subgroups $G_1,\ldots,G_k$ are disjoint. 
Then there exist absolute constants $c,\, C>0$ such that for all $s>0,$ $n$ and $k$ satisfying
\begin{equation}
\label{eq:assump}
\frac{1}{\Delta} \l( \frac{1}{\sqrt{k}} \, \mb E\sup_{f\in \m F} \frac{1}{\sqrt{N}}\sum_{j=1}^N \l( f(X_j) - Pf \r) + \sigma(\m F)\sqrt{\frac{s}{k}} \r) +
\sup_{f\in\m F} G_f(n,\Delta)
+ \frac{s}{k} \leq c,
\end{equation}
the following inequality holds with probability at least $1 - 2e^{-s}$:
\begin{multline*}
\sup_{f\in \m F}\l| \wh \theta^{(k)}(f) - Pf \r| \leq 
C
\Bigg[ \frac{\widetilde \Delta}{\Delta} \l( \mb E\sup_{f\in \m F} \frac{1}{N} \sum_{j=1}^N \l( f(X_j) - Pf \r)
+ \sigma(\m F) \sqrt{\frac{s}{N}} \r)
\\
+ \widetilde \Delta \l( \sqrt{n}\frac{s}{N}
+ \frac{\sup_{f\in\m F} G_f(n,\Delta)}{\sqrt{n}} \r) \Bigg].
\end{multline*}
\end{theorem}
\noindent The proof of the theorem is outlined in section \ref{proof:unif}. 
We note that the requirement $N=nk$ is not restrictive, as replacing $N$ by $k\cdot \lfloor N/k \rfloor$ will only result in the change of absolute constants. 
The assumption that the groups sizes are equal is also not essential and is only imposed to avoid overly technical and cumbersome expressions. 
When the class $\m F$ is P-Donsker \cite{dudley2014uniform},  
$\limsup\limits_{N\to\infty} \Big| \mb E\sup\limits_{f\in \m F} \frac{1}{\sqrt N} \sum_{j=1}^N \l( f(X_j) - Pf \r) \Big|$ is bounded, hence condition \eqref{eq:assump} holds for $N$ large enough whenever $s$ is not too big and $\Delta$ is not too small, namely, $s\leq c'k$ and $\Delta \geq c'' \sigma(\m F)$. When discussing examples in the following section, this will be our default setup.

Estimator $\wh\theta^{(k)}(f)$ defined in \eqref{eq:M-est} depends on the choice of subgroups $G_1,\ldots,G_k$. It is natural to ask if there exists a version of $\widehat \theta^{(k)}(f)$ that is permutation-invariant. We address this question below and present a construction based on U-statistics.  
For an integer $n\leq \frac{N}{2}$, let $k=\lfloor N/n\rfloor$, and define 
\[
\m A_N^{(n)}:=\l\{  J: \ J\subseteq \{1,\ldots,N\}, \card(J)=n \r\}.
\] 
Let $h$ be a measurable function of $n$ variables. Recall that a U-statistic of order $n$ with kernel $h$ based on the i.i.d. sample $X_1,\ldots,X_N$ is defined as \citep{hoeffding1948class}
\begin{equation}
\label{u-stat}
U_{N,n} = \frac{1}{{N\choose n}}\sum_{J\in \m A_N^{(n)}} h \l( \{X_j\}_{j \in J} \r).
\end{equation}
Given $J\in A_N^{(n)}$, let $\bar\theta(f; J):=\frac{1}{n}\sum_{i\in J} f(X_i)$. Consider U-statistics of the form 
\[
U_{N,n}(z;f) = \frac{1}{{N\choose n}} \sum_{J\in \m A_N^{(n)}} \rho\l(\sqrt{n}\,\frac{\bar \theta(f;J) - z}{\Delta}\r),
\]
and set
\begin{equation*}
\widetilde\theta^{(k)}(f):= \argmin_{z\in \mb R} U_{N,n}(z;f).
\end{equation*}
\begin{theorem}
\label{th:unif-U}
There exist absolute constants $c,\, C>0$ such that for all $s>0,$ $n$ and $k$ satisfying
\begin{equation}
\label{eq:assump}
\frac{1}{\Delta} \l( \frac{1}{\sqrt{k}} \, \mb E\sup_{f\in \m F} \frac{1}{\sqrt{N}}\sum_{j=1}^N \l( f(X_j) - Pf \r) + \sigma(\m F)\sqrt{\frac{s}{k}} \r) +
\sup_{f\in\m F} G_f(n,\Delta)
+ \frac{s}{k} \leq c,
\end{equation}
the following inequality holds with probability at least $1 - 2e^{-s}$:
\begin{multline*}
\sup_{f\in \m F}\l| \wh \theta^{(k)}(f) - Pf \r| \leq 
C
\Bigg[ \frac{\widetilde \Delta}{\Delta} \l( \mb E\sup_{f\in \m F} \frac{1}{N} \sum_{j=1}^N \l( f(X_j) - Pf \r)
+ \sigma(\m F) \sqrt{\frac{s}{N}} \r)
\\
+ \widetilde \Delta \l( \sqrt{n}\frac{s}{N}
+ \frac{\sup_{f\in\m F} G_f(n,\Delta)}{\sqrt{n}} \r) \Bigg].
\end{multline*}

\end{theorem}
The resulting deviation bounds for $\widetilde \theta^{(k)}(f)$ are of exactly the same form as for the estimator $\wh\theta^{(f)}(f)$ that is based on disjoint blocks of data. The proof of this result is given in section \ref{proof:unif-U}.

\subsection{Discussion and comparison with existing bounds.}
\label{sec:comparison}

A number of recent works address the problem of robust empirical risk minimization that is closely related to the question addressed in the present paper. 
In \cite{brownlees2015empirical}, authors prove uniform deviation bounds for robust mean estimators defined using O. Catoni's approach \cite{catoni2012challenging}; these bounds are limited by their dependence on the covering numbers of the class $\m F$ with respect to the sup-norm $\l\| \cdot \r\|_\infty$. 
  Uniform bounds for the median-of-means estimators have been obtained in several papers, including \cite{lugosi2016risk,lecue2018robust,lugosi2017regularization,lecue2017robust}. 
Result that is closest to our setting has been obtained in \cite{lecue2018robust}: proof of Theorem 2 of that paper implies that the estimator $\wh\theta^{(k)}_{\mathrm{med}}(f)$ corresponding to $\rho(x)=|x|$ satisfies
\begin{equation}
\label{eq:slow}
\sup_{f\in \m F}\l| \wh\theta^{(k)}_{\mathrm{med}}(f) - Pf \r| \leq 
C
\Bigg( \frac{1}{\sqrt{N}} \, \mb E\sup_{f\in \m F} \l| \frac{1}{\sqrt{N}} \sum_{j=1}^N \l( f(X_j) - Pf \r)\r| 
+ \sigma(\m F) \sqrt{\frac{k}{N}} \Bigg)
\end{equation}
with probability at least $1 - e^{-ck}$ for absolute constant $c,C>0$; a slightly stronger version of this result has appeared in \cite{lugosi2018near}. 
The key difference between this inequality and the bound of Theorem \ref{th:unif} is the fact that the former holds only for the fixed value of the confidence parameter $s=k$, while the latter typically provides deviation guarantees over the wide range  $0<s\leq ck$ of confidence parameter $s$. 
This difference is important, as one usually wants to choose $k$ as large as possible to improve robustness (in particular, robustness to adversarial contamination) without degrading performance of the estimator. 
The price that we have to pay is the necessity to tune the parameter $\Delta$. However, as we show below, in many cases there is a wide range of ``suitable'' choices of $\Delta$, so this issue is not critical.

Let us consider two examples. 
First, assume that $n=1$, $k=N$ and set $\Delta = \Delta(s):= \sigma(\m F)\sqrt{\frac{N}{s}}$. In this case, for $N$ large enough, condition \eqref{eq:assump} reduces to $s\leq c' N$, and we deduce from Theorem \ref{th:unif} that 
\[
\sup_{f\in \m F}\l| \wh \theta^{(k)}(f) - Pf \r| \leq 
C
\Bigg[ \mb E\sup_{f\in \m F} \frac{1}{N} \sum_{j=1}^N \l( f(X_j) - Pf \r)
+ \sigma(\m F) \sqrt{\frac{s}{N}} \Bigg]
\] 
with probability at least $1-2e^{-s}$. 
This inequality recovers, up to constants, the result of Catoni \cite{catoni2012challenging} for $\m F=\{ f\}$ (indeed, in this case the expected supremum is $0$), and improves upon uniform bounds obtained for Catoni-type estimators in \cite{brownlees2015empirical}. If $\Delta = \sigma(\m F)\sqrt{N}$, we get the ``sub-exponential'' bound 
\[
\sup_{f\in \m F}\l| \wh \theta^{(k)}(f) - Pf \r| \leq 
C
\Bigg[ \mb E\sup_{f\in \m F} \frac{1}{N} \sum_{j=1}^N \l( f(X_j) - Pf \r)
+ \sigma(\m F) \frac{1+s}{\sqrt{N}} \Bigg]
\]
that holds with probability at least $1-e^{-s}$ uniformly for $s\leq c' N$. 
Moreover, if 
\[
\kappa_{2+\delta}(\m F):=\frac{\sup_{f\in \m F} \mb E|f(X) - \mb Ef(X)|^{2+\delta}}{\sigma^{2+\delta}(\m F)}<\infty
\] 
for some $\delta\in(0,1]$, then the estimator is less sensitive to the choice of the parameter $\Delta$. Specifically, 
\begin{equation}
\label{eq:corollary}
\sup_{f\in \m F}\l| \wh \theta^{(k)}(f) - Pf \r| \leq 
C
\Bigg[ \mb E\sup_{f\in \m F} \frac{1}{N} \sum_{j=1}^N \l( f(X_j) - Pf \r)
+ \sigma(\m F) \sqrt{\frac{s}{N}} \Bigg]
\end{equation}
with probability at least $1-2e^{-s}$ for any $\Delta$ satisfying 
$\sigma(\m F)\l( \frac{N}{s}\r)^{\frac{1}{2(1+\delta)}} \l(\kappa_{2+\delta}(\m F)\r)^{1/(1+\delta)} \leq \Delta \leq \sigma(\m F)\sqrt{\frac{N}{s}}$. For instance, if $\delta=1$, then any $\Delta$ in the range $\sigma(\m F)(N/s)^{1/4}\lesssim \Delta \lesssim \sigma(\m F)\sqrt{N/s}$ is suitable.
Equivalently, for a given $\Delta$, the sub-Gaussian type bound \eqref{eq:corollary} holds uniformly for all $s \in \l[ \kappa^2_{3} \frac{M_\Delta^{4}}{N}, M_\Delta^2\r]$ where $M_\Delta = \frac{\sigma(\m F)\sqrt{N}}{\Delta}$. 

\noindent Next, assume that $N\gg n\geq 2$. For $\Delta = \sigma(\m F)\sqrt{\frac k s}$, we again recover the bound 
\[
\sup_{f\in \m F}\l| \wh \theta^{(k)}(f) - Pf \r| \leq 
C
\Bigg[ \mb E\sup_{f\in \m F} \frac{1}{N} \sum_{j=1}^N \l( f(X_j) - Pf \r)
+ \sigma(\m F) \sqrt{\frac{s}{N}} \Bigg]
\]
that holds with probability at least $1-2e^{-s}$. 
When $\Delta\asymp \sigma(\m F)$, $\widehat \theta^{(k)}(f)$ most closely resembles the median-of-means estimator. 
In this case, the inequality that holds with probability at least $1-2e^{-s}$ is  
\[
\sup_{f\in \m F}\l| \wh \theta^{(k)}(f) - Pf \r| \leq 
C
\Bigg[ \mb E\sup_{f\in \m F} \frac{1}{N} \sum_{j=1}^N \l( f(X_j) - Pf \r)
+ \sigma(\m F) \l(\sqrt{\frac{s}{N}} + \sqrt{\frac k N}\sup_{f\in \m F} G_f(n,\sigma(\m F)) \r) \Bigg].
\]
As $\sup_{f\in \m F} G_f(n,\sigma(\m F))$ is small for large $n$, this bound is clearly better than \eqref{eq:slow}. 
Finally, let us again consider the case when stronger moment assumptions hold, namely, $\kappa_3(\m F)<\infty$. 
Combining the estimate for $\sup_{f\in \m F}G_f(n,\Delta)$ provided by Lemma \ref{lemma:BE-nonunif} and the bound of Theorem \ref{th:unif}, we obtain that 
\begin{multline*}
\sup_{f\in \m F}\l| \wh \theta^{(k)}(f) - Pf \r| \leq 
C
\Bigg[ \mb E\sup_{f\in \m F} \frac{1}{N} \sum_{j=1}^N \l( f(X_j) - Pf \r)
+ \sigma(\m F) \sqrt{\frac{s}{N}}
\\
+ \frac{\sup_{f\in \m F}\mb E|f(X) - \mb Ef(X)|^3}{\Delta^2} \frac{k}{N} + \Delta \frac{s}{\sqrt{kN}}\Bigg],
\end{multline*}
again with probability at least $1-2e^{-s}$. 
If $\Delta$ is such that $c_1\sigma(\m F)\l(1\vee\sqrt{\frac{k}{\sqrt{Ns}}} \r) \leq \Delta \leq c_2\sigma(\m F)\sqrt{k/s}$ then this inequality implies sub-Gaussian deviation bounds at confidence level $s$.  
If for instance $1\ll k\ll \sqrt{N}$ and $\Delta \asymp \sigma(\m F)$, then 
\[
\sup_{f\in \m F}\l| \wh \theta^{(k)}(f) - Pf \r| \leq 
C
\Bigg[ \mb E\sup_{f\in \m F} \frac{1}{N} \sum_{j=1}^N \l( f(X_j) - Pf \r)
+ \sigma(\m F) \sqrt{\frac{s}{N}} \Bigg]
\]
for all $s\leq c'k$ uniformly. 
\section{Contamination with outliers.}
\label{sec:adversary}


Assume that the initial dataset of cardinality $N$ is merged with a set of $\m O<N$ outliers that are generated by an adversary who has an opportunity to inspect the data, and the combined dataset of cardinality $N^\circ = N+\m O$ is presented to a statistician. 
We would like to understand performance of proposed estimators $\widehat \theta^{(k)}(f)$ in this more challenging framework. 
Let $G_1,\ldots,G_k$ be the disjoint partition of the index set $\{1,\ldots,N^\circ\}$ that the estimators $\l\{\wh \theta^{(k)}(f), \ f\in \m F\r\}$ are based on; as before, $n$ will stand for the cardinality of $G_j$. 

In the rest of the section, we will assume that $k>2\m O$. 
Let $J\subset \{1,\ldots,k\}$ of cardinality $|J|\geq k-\m O$ be the subset containing all $j$ such that the subsample $\{ X_i, \ j\in G_j\}$ does not include outliers. Clearly, $\{ X_i: \ i\in G_j, \ j\in J\}$ are still i.i.d. as the partitioning scheme is independent of the data. 
Moreover, set $N_J:=\sum_{j\in J} |G_j|$, and note that 
\[
N_J\geq n |J|\geq \frac{kn}{2}.
\] 
The following analogue of Theorem \ref{th:unif} holds.
\begin{theorem}
\label{th:unif2}
There exist absolute constants $c,\, C>0$ such that for all $s>0,$ $n$ and $k$ satisfying
\begin{equation}
\label{eq:assump}
\frac{1}{\Delta} \l( \frac{1}{\sqrt{k}} \, \mb E\sup_{f\in \m F} \frac{1}{\sqrt{N}}\sum_{j=1}^N \l( f(X_j) - Pf \r) + \sigma(\m F)\sqrt{\frac{s}{k}} \r) +
\sup_{f\in\m F} G_f(n,\Delta)
+ \frac{s+\m O}{k} \leq c,
\end{equation}
the following inequality holds with probability at least $1 - 2e^{-s}$:
\begin{multline*}
\sup_{f\in \m F}\l| \wh \theta^{(k)}(f) - Pf \r| \leq 
C
\Bigg[ \frac{\widetilde \Delta}{\Delta} \l( \mb E\sup_{f\in \m F} \frac{1}{N} \sum_{j=1}^N \l( f(X_j) - Pf \r)
+ \sigma(\m F) \sqrt{\frac{s}{N}} \r)
\\
+ \widetilde \Delta \l( \sqrt{n}\frac{s+ \m O}{N}
+ \frac{\sup_{f\in\m F} G_f(n,\Delta)}{\sqrt{n}} \r) \Bigg].
\end{multline*}
\end{theorem}
It is convenient to interpret the inequality as follows: if $\m O\leq c k$ for a sufficiently small absolute constant $c$, the error $\sup_{f\in \m F}\l| \wh \theta^{(k)}(f) - Pf \r|$ behaves like the maximum of 2 terms: the first term is the error bound for the case $\m O = 0$, and the second term is of order 
$\widetilde\Delta \sqrt{n}\frac{\m O}{N}$. In the next section, we provide examples which show that the dependence on $\m O$ in our bounds is, in general, non-improvable. 





\subsection{Estimators of the mean of a random vector.}
\label{section:mean}

Assume that $X_1,\ldots,X_N$ are i.i.d. copies of a random vector $X\in \mb R^d$ with mean $\mb EX = \mu$ and covariance matrix $\mb E(X-\mu)(X-\mu)^T = \Sigma$. 
Let $\|\cdot\|$ be some norm in $\mb R^d$, and let $B$ be the unit ball with respect to this norm, $B=\l\{ x\in \mb R^d: \ \|x\|\leq 1\r\}$. Consider the class of linear functionals $\m F=\l\{ f_v(x)=\langle v,x \rangle, \ v\in B\r\}$. 
Our goal is to estimate the mean $\mu$, with the error measured in the norm $\|\cdot\|$. 
Construction that we propose is closely related to the approach employed previously by several authors \cite{joly2017estimation,lugosi2018near,catoni2016pac,giulini2018robust} that is based on combining estimators of one-dimensional projections. 
Assume that we are in the ``adversarial contamination'' framework of Theorem \ref{th:unif2}. 
Let $\rho$ be a function satisfying assumption \ref{ass:1}, and let $\wh\theta^{(k)}(v)$ be the estimator of $\langle \mu,v \rangle$, the projection of $\mu$ in direction $v\in B$:  
\[
\wh\theta^{(k)}(v):= \argmin_{z\in \mb R}\frac{1}{\sqrt{N}}\sum_{j=1}^k \rho\l(\sqrt{n}\,\frac{\bar \mu_j(v) - z}{\Delta}\r),
\]
where $\bar\mu_j(v)=\frac{1}{n}\sum_{i\in G_j}\langle v,X_j\rangle$, and $\Delta\geq \sqrt{\lambda_{\max}(\Sigma)}$. 
Given $v\in B$ and $\eps>0$, define the closed ``slab''
\[
S_v(\eps):=\l\{ y\in \mb R^d: \ \l| \langle y, v \rangle - \wh\theta^{(k)}(v) \r| \leq \eps\r\}, 
\] 
and $M(\eps):=\bigcap_{v\in B} S_v(\eps)$. 
Finally, let $\eps_\ast:=\inf\l\{ \eps>0: \ M(\eps)\ne \emptyset\r\}$, and 
take $\wh\mu^{(k)}$ to be any element in $M(\eps_\ast)$ (indeed, $M(\eps_\ast) = \bigcap_{\eps > \eps_\ast} M(\eps)$ is non-empty as an intersection of nested compact sets). 
\begin{corollary}
\label{cor:norm}
There exist absolute constants $\tilde c,\tilde C>0$ with the following properties: assume that 
\begin{equation}
\label{cond:2}
\max\l( \frac{1}{\sqrt{k}\Delta}  \, \mb E\sup_{v\in B} \l| \frac{1}{\sqrt{N}}  \sum_{j=1}^N \langle v, X_j - \mu \rangle \r| + \frac{s}{k} + \sup_{v\in B} G_f(n,\Delta), \frac{\m O}{k}   \r) \leq \tilde c.
\end{equation}
Then with probability at least $1-2e^{-s}$, 
\[
\l\| \wh\mu^{(k)} - \mu\r\| \leq \tilde C \Bigg(  \mb E\sup_{v\in B}  \frac{1}{N} \sum_{j=1}^N \langle v, X_j - \mu \rangle + 
\sqrt{\lambda_{\max}(\Sigma)}\sqrt{\frac{s}{N}} + \Delta\l( \sup_{v\in B}  \frac{G_{f_v}(n,\Delta)}{\sqrt n} + \sqrt{n}\frac{s+\m O}{N}\r)
 \Bigg).
\]
\end{corollary}
\begin{proof}
It follows from Theorem \ref{th:unif2} that on the event $\m E$ of probability at least $1-2e^{-s}$, $\mu\in M(\eps)$ for all 
\[
\eps\geq \eps_0:=C \Bigg( \mb E\sup_{f\in \m F}  \frac{1}{N}  \sum_{j=1}^N \langle v, X_j - \mu \rangle + 
\sqrt{\lambda_{\max}(\Sigma)}\sqrt{\frac{s}{N}} + \Delta\l( \sup_{v\in B}  \frac{G_{f_v}(n,\Delta)}{\sqrt n} + \sqrt{n}\frac{s+\m O}{N}\r)
 \Bigg)
\] 
given that \eqref{cond:2} holds, hence $\eps_\ast \leq \eps_0$ on event $\m E$. Consequently, 
\[
\l\| \wh\mu^{(k)} - \mu \r\| = \sup_{v\in B}\l| \langle \wh\mu^{(k)} - \mu,v\rangle \r|
\leq \sup_{v\in B}\l| \langle \wh\mu^{(k)} ,v\rangle - \wh\theta^{(k)}(v) \r| + \sup_{v\in B}\l| \langle \mu ,v\rangle -  \wh\theta^{(k)}(v)\r| \leq \eps_\ast + \eps_0 \leq 2\eps_0
\]
with probability at least $1 - 2e^{-s}$.
\end{proof}

In the special case when $\|\cdot\|$ is the Euclidean norm $\|\cdot\|_2$, the bound of Corollary \ref{cor:norm} can be further simplified. It follows from H\"{o}lder's inequality that 
\begin{multline*}
\mb E\sup_{v\in \mb R^d: \|v\|_2=1} \l| \frac{1}{\sqrt{N}}  \sum_{j=1}^N \langle v, X_j - \mu \rangle \r| 
\leq \mb E^{1/2}\sup_{v\in \mb R^d: \|v\|_2=1} \l| \frac{1}{\sqrt{N}}  \sum_{j=1}^N \langle v, X_j - \mu \rangle \r|^2 
\\
= \sqrt{\frac{1}{N} \mb E \l\| \sum_{j=1}^N (X_j - \mu)\r\|_2^2} = \sqrt{\mb E \l\| X - \mu \r\|_2^2} =\sqrt{\tr \Sigma},
\end{multline*}
hence $\wh \mu^{(k)}$ satisfies 
\[
\l\| \wh\mu^{(k)} - \mu\r\|_2 \leq C \Bigg( \sqrt{\frac{\tr \Sigma}{N}} + \sqrt{\lambda_{\max}(\Sigma)}\sqrt{\frac{s}{N}} + \Delta\l( \sup_{v\in B}  \frac{G_{f_v}(n,\Delta)}{\sqrt n} + \sqrt{n}\frac{s+\m O}{N}\r) \Bigg)
\]
with probability at least $1 - 2e^{-s}$ (whenever $s\leq c'k$). 
According to the discussion in section \ref{sec:comparison}, in many cases the term 
$\Delta\l( \sup_{v\in B}  \frac{G_{f_v}(n,\Delta)}{\sqrt n} + \sqrt{n}\frac{s+\m O}{N} \r)$ is of order smaller than $N^{-1/2}$, whence the estimator $\widehat\mu^{(k)}$ behaves like the sample mean of the Gaussian random variables \cite{lugosi2018near}. 

Let us now discuss optimality with respect to adversarial contamination. Assume that $\m O = \eps N$ for $\eps\geq \frac{1}{N}$; here, we assume $\eps$  to be known in advance, and the issue of adaptivity is beyond the scope of this paper. 
 Moreover, suppose that 
\[
\kappa_{2+\delta}:=\sup_{v: \|v\|_2=1}\frac{\mb E \l| \langle v,X - \mu\rangle \r|^{2+\delta}}{\l( \var\langle v,X-\mu\rangle\r)^{1+\delta/2}} < \infty
\]
for some $\delta\in(0,1]$. In this case, Lemma \ref{lemma:BE-nonunif} implies that 
\[
\sup_{v\in B} G_{f_v}(n,\Delta)\leq C \frac{\sup_{\|v\|_2=1} \mb E\l| \dotp{X-\mu}{v}\r|^{2+\delta}}{\Delta^{2+\delta}n^{\delta/2}}
\]
for some absolute constant $C>0$. 
Let $M_\Delta:=\frac{\Delta}{\sigma(\m F)}$, and recall that $M_\Delta\geq 1$ by assumption. Then, choosing 
$k = \eps^{\frac{2}{2+\delta}}\, N\frac{M_\Delta^2 }{\kappa^{\frac{2}{2+\delta}}_{2+\delta}}$, we deduce from Corollary \ref{cor:norm} that 
\[
\l\| \wh\mu^{(k)} - \mu\r\|_2 \leq C \Bigg( \sqrt{\frac{\tr \Sigma}{N}} + \sqrt{\lambda_{\max}(\Sigma)}\l(\sqrt{\frac{s}{N}} + \eps^{\frac{1+\delta}{2+\delta}}\, \kappa^{1/(2+\delta)}_{2+\delta} \r) + \frac{\kappa^{1/(2+\delta)}_{2+\delta}}{\eps^{1/(2+\delta)} N}\, s \Bigg)
\]
with probability at least $1-2e^{-s}$. Since $\eps\geq \frac{1}{N}$, $\eps^{1/(2+\delta)} N \geq N^{-\frac{1+\delta}{2+\delta}} = o\l( N^{-1/2} \r)$ whenever $\delta>0$, hence the last term in the bound is of smaller order. 
According to Lemma \ref{lemma:minmax} (see section \ref{section:supplement}), no estimator can achieve rate faster than $\eps^{\frac{1+\delta}{2+\delta}}$ with respect to $\eps$, implying that our estimator is optimal in this sense.


\section{Proofs.}
\label{sec:proofs}

We will introduce some additional notation and recall useful results that we rely upon in the proofs. 
Denote 
\[
G_k(z;f) = \frac{1}{\sqrt{k}}\sum_{j=1}^k \rho'\l( \sqrt{n}\,\frac{ (\bar \theta_j(f) - Pf) - z}{\Delta} \r)
\]
and let $\wh \theta_0^{(k)}(f)$ be any solution of the equation $G_k\l(\wh \theta_0^{(k)}(f);f\r)=0$. 
Clearly, $\wh \theta_0^{(k)}(f) = \wh \theta^{(k)}(f) - Pf$ is the error of the estimator $\wh \theta^{(k)}(f)$.

\noindent The following concentration inequality is due to Klein and Rio (see section 12.5 in \cite{boucheron2013concentration}). 
\begin{fact}
\label{fact:ineq}
Let $\{Z_j(f), \ f\in \m F\}, \ j=1,\ldots,N$ be independent (not necessarily identically distributed) separable stochastic processes indexed by class $\m F$ and such that $|Z_j(f)-\mb EZ_j(f)|\leq M$ a.s. for all $1\leq j\leq N$ and $f\in \m F$. Then the following inequality holds with probability at least $1-e^{-s}$:
\begin{align}
\label{eq:klein1}
\sup_{f\in \m F}\l( \sum_{j=1}^N (Z_j(f) - \mb EZ_j(f)) \r) &\leq 2\mb E\sup_{f\in \m F} \l( \sum_{j=1}^N (Z_j(f) - \mb EZ_j(f))\r) 
+V(\m F)\sqrt{2s}+\frac{4Ms}{3},
\end{align}
where $V^2(\m F)=\sup_{f\in \m F}\sum_{j=1}^N \var\l( Z_j(f)\r)$.
\end{fact}
\noindent It is easy to see, applying \eqref{eq:klein1} to processes $\{-Z_j(f), \ f\in \m F\}$, that
\begin{equation}
\label{eq:klein2}
\inf_{f\in \m F}\l( \sum_{j=1}^N (Z_j(f) - \mb EZ_j(f))\r) \geq -2\mb E\sup_{f\in \m F} \l( \sum_{j=1}^N (\mb EZ_j(f) - Z_j(f))\r) 
 - V(\m F)\sqrt{2s} - \frac{4Ms}{3}
\end{equation}
with probability at least $1-e^{-s}$.	
Moreover, \eqref{eq:klein1} is a corollary of the following bound for the moment generating function: 
\begin{equation}
\label{eq:mgf}
\log \mb E e^{\lambda \l( \sum_{j=1}^N (Z_j(f) - \mb EZ_j(f))\r)}\leq \frac{e^{\lambda M} - \lambda M - 1}{M^2}\l( V^2(\m F) + 2M \,\mb E\sup_{f\in \m F} \l( \sum_{j=1}^N (Z_j(f) - \mb E Z_j(f))\r)\r)
\end{equation}
that holds for all $\lambda>0$. 
This fact provides a straightforward extension of the concentration bounds to the case of U-statistics. 
Let $\pi_N$ be the collection of all permutations $i:\{1,\ldots,N\}\mapsto \{1,\ldots,N\}$. 
Given $(i_1,\ldots,i_N)\in \pi_N$ and a U-statistic $U_{N,n}$ defined in \eqref{u-stat}, let 
\begin{equation*}
T_{i_1,\ldots,i_N}:=\frac{1}{k}\l( h\l(X_{i_1},\ldots,X_{i_n} \r) + h\l(X_{i_{n+1}},\ldots,X_{i_{2n}}\r) + \ldots + 
h\l(X_{i_{(k-1)n+1}},\ldots,X_{i_{kn}} \r) \r).
\end{equation*}
It is well known (see section 5 in \cite{hoeffding1963probability}) that the following representation holds:
\begin{equation}
\label{eq:U-decomp}
U_{N,n} = \frac{1}{N!}\sum_{(i_1,\ldots,i_N) \in \pi_N}  T_{i_1,\ldots,i_N}.
\end{equation}
Let $U_{N,n}'(z;f) = \frac{1}{{N\choose n}} \sum_{J\in \m A_N^{(n)}} \rho'\l(\sqrt{n}\,\frac{(\bar \theta(f;J) - Pf) - z}{\Delta}\r)$. 
Applied to $U'_{N,n}(z;f)$, relation \eqref{eq:U-decomp} yields that 
\[
U'_{N,n}(z;f) = \frac{1}{N!}\sum_{(i_1,\ldots,i_N)\in \pi_N} T_{i_1,\ldots,i_N}(z;f),
\]
where 
\begin{multline*}
T_{i_1,\ldots,i_N}(z;f) = \frac{1}{k}\Big(  \rho'\l(\sqrt{n}\,\frac{\bar \theta(f;\{i_1,\ldots,i_n\}) - Pf - z}{\Delta}\r) +
\\ 
\ldots +  \rho'\l(\sqrt{n}\,\frac{\bar \theta(f;\{i_{(k-1)n+1},\ldots,i_{kn}\}) - Pf - z}{\Delta}\r) \Big).
\end{multline*}
Jensen's inequality implies that for any $\lambda>0$,
\begin{multline*}
\mb E \exp\l( \frac{\lambda}{N!} \sum_{(i_1,\ldots,i_N)\in \pi_N} \l( T_{i_1,\ldots,i_N}(z;f) -\mb E T_{i_1,\ldots,i_N}(z;f)\r)\r) 
\\
\leq \frac{1}{N!} \sum_{(i_1,\ldots,i_N)\in \pi_N} \mb E \exp\Big( \lambda \l( T_{1,\ldots,N}(z;f) -\mb E T_{1,\ldots,N}(z;f)\r) \Big),
\end{multline*}
hence bound \eqref{eq:mgf} can be applied and yields that 
\begin{multline}
\label{eq:U-concentr}
\sup_{f\in \m F}\l( U'_{N,n}(z;f) - \mb U'_{N,n}(z;f)\r) \leq 2\mb E\sup_{f\in \m F}\l( T_{1,\ldots,N}(z;f) - \mb E T_{1,\ldots,N}(z;f) \r) 
\\
+ \sup_{f\in \m F}\sqrt{\var\l( \rho'\l(\sqrt{n}\,\frac{\bar \theta(f;\{1,\ldots,n\}) - Pf - z}{\Delta}\r)\r)}\sqrt{\frac{2s}{k}} + \frac{8s\|\rho'\|_\infty}{3k}
\end{multline}
with probability at least $1-e^{-s}$. The expression can be further simplified by noticing that $\|\rho'\|_\infty\leq 2$ and that  
\[
\var\l( \rho'\l(\sqrt{n}\,\frac{\bar \theta(f;\{1,\ldots,n\}) - Pf - z}{\Delta}\r)\r) \leq \frac{\sigma^2(f)}{\Delta^2}.
\]
due to Lemma \ref{lemma:variance}. 

\subsection{Proof of Theorem \ref{th:unif}.}
\label{proof:unif}

Recall that
\[
G_k(z;f) = \frac{1}{\sqrt{k}}\sum_{j=1}^k \rho'\l( \sqrt{n}\,\frac{ (\bar \theta_j(f) - Pf) - z}{\Delta} \r).
\]
Suppose $z_1,z_2$ are such that on an event of probability close to $1$, $G_k(z_1;f) > 0$ and $G_k(z_2;f) < 0$ for all $f\in \m F$ simultaneously. 
Since $G_k$ is decreasing in $z$, it is easy to see that $\wh \theta_0^{(k)}(f)\in (z_1,z_2)$ for all $f\in \m F$ on this event, implying that $\l| \sup_{f\in \m F}\wh \theta_0^{(k)}(f) \r|\leq \max(|z_1|,|z_2|)$. 
Hence, our goal is to find $z_1,z_2$ satisfying conditions above and such that $|z_1|,\, |z_2|$ are as small as possible. 
We will provide detailed bounds for $z_1$, while the steps to estimate $z_2$ are quite similar. 
Observe that 
\begin{multline*}
G_k(z;f) = \frac{1}{\sqrt{k}} \sum_{j=1}^k \l(  \rho'\l( \sqrt{n}\frac{(\bar \theta_j(f) - Pf) - z}{\Delta} \r) - \mb E\rho'\l(\sqrt{n} \frac{(\bar \theta_j(f) - Pf) - z}{\Delta} \r)  \r)  \\ 
+ \frac{1}{\sqrt{k}}\sum_{j=1}^k  \l( \mb E\rho'\l(\sqrt{n} \frac{(\bar \theta_j(f) - Pf) - z}{\Delta} \r) - 
\mb E \rho'\l( \frac{W(f) - \sqrt{n} z }{\Delta} \r) \r) + 
\frac{1}{\sqrt{k}}\sum_{j=1}^k \mb E \rho'\l( \frac{W(f) - \sqrt{n} z }{\Delta} \r).
\end{multline*}
We will proceed in 3 steps: first, we will find $\eps_1>0$ such that for any $z\in \mb R$,
\begin{equation}
\label{eq:eps1}
\inf_{f\in \m F} \frac{1}{\sqrt{k}}\sum_{j=1}^k \l( \rho'\l( \sqrt{n}\frac{(\bar \theta_j(f) - Pf) - z}{\Delta} \r) - \mb E\rho'\l(\sqrt{n} \frac{(\bar \theta_j(f) - Pf) - z}{\Delta} \r)  \r) \geq -\eps_1
\end{equation}
with high probability, then $\eps_2>0$ such that 
\begin{equation*}
\inf_{f\in \m F} \frac{1}{\sqrt{k}}\sum_{j=1}^k\l( \mb E\rho'\l(\sqrt{n} \frac{(\bar \theta_j(f) - Pf) - z}{\Delta} \r) - 
\mb E \rho'\l( \frac{W(f) - \sqrt{n} z }{\Delta} \r) \r) \geq -\eps_2,
\end{equation*}
and finally we will choose $z_1<0$ such that for all $f\in \m F$,
\begin{equation}
\label{eq:z1}
\frac{1}{\sqrt{k}} \sum_{j=1}^k \mb E \rho'\l( \frac{W(f) - \sqrt{n} z_1 }{\Delta} \r) > \eps_1 + \eps_2.
\end{equation}
It follows from Lemma \ref{lemma:step1} that setting 
\[
\eps_1 =  \frac{8 }{\Delta\sqrt{N}} \, \mb E\sup_{f\in \m F} \sum_{j=1}^N \l( f(X_j) - Pf \r) 
+ \frac{\sigma(\m F)}{\Delta}\sqrt{2s} + \frac{16}{3}\frac{s}{\sqrt{k}}
\]
guarantees that \eqref{eq:eps1} holds with probability at least $1 - e^{-s}$. 
Next, Lemma \ref{lemma:step2} implies that $\eps_2$ can be chosen as
\[
\eps_2 = 2\sqrt{k}\,\sup_{f\in\m F} G_f(n,\Delta).
\]
Finally, we apply Lemma \ref{lemma:step3} with 
\[
\eps := \frac{\eps_1 + \eps_2}{\sqrt{k}}
\]
to deduce that 
\begin{equation*}
z_1 = -  \frac{1}{0.09}
 \bigg( \frac{8 L(\rho')}{N}\frac{\widetilde \Delta}{\Delta} \, \mb E\sup_{f\in \m F}  \sum_{j=1}^N \l( f(X_j) - Pf \r)\ + \sigma(\m F)\frac{\widetilde \Delta}{\Delta}\sqrt{\frac{2s}{N}} + \widetilde \Delta\frac{16}{3} \frac{s\sqrt{n}}{N} + 2\frac{\sup_{f\in \m F} G_f(n,\Delta)}{\sqrt{n}} \Bigg)
\end{equation*}
satisfies \eqref{eq:z1} under assumption that $\eps\leq 0.045$. 
Proceeding in a similar way, it is easy to see that setting $z_2=-z_1$ guarantees that $G_k(z_2;f)<0$ for all $f\in \m F$ with probability at least $1 - e^{-s}$, hence the claim follows.

\begin{lemma}
\label{lemma:step1}
For any $z\in \mb R$ and $\rho$ satisfying Assumption \ref{ass:1}, the inequalities
\begin{align*}
\inf_{f\in \m F}\frac{1}{\sqrt{k}} \sum_{j=1}^k &\l(  \rho'\l( \sqrt{n}\frac{(\bar \theta_j(f) - Pf) - z}{\Delta} \r) - \mb E\rho'\l(\sqrt{n} \frac{(\bar \theta_j(f) - Pf) - z}{\Delta} \r)  \r)  \\
&\geq -\frac{8}{\Delta\sqrt{N}} \, \mb E\sup_{f\in \m F} \sum_{j=1}^N \l( f(X_j) - Pf \r) 
- \frac{\sigma(\m F)}{\Delta}\sqrt{2s} - \frac{16}{3} \frac{s}{\sqrt{k}},
\\
\sup_{f\in \m F}\frac{1}{\sqrt{k}} \sum_{j=1}^k &\l(  \rho'\l( \sqrt{n}\frac{(\bar \theta_j(f) - Pf) - z}{\Delta} \r) - \mb E\rho'\l(\sqrt{n} \frac{(\bar \theta_j(f) - Pf) - z}{\Delta} \r)  \r)  \\
&\leq \frac{8}{\Delta\sqrt{N}} \, \mb E\sup_{f\in \m F} \sum_{j=1}^N \l( f(X_j) - Pf \r) 
+ \frac{\sigma(\m F)}{\Delta}\sqrt{2s} + \frac{16}{3}\frac{s}{\sqrt{k}},
\end{align*}
hold with probability at least $1-e^{-s}$ each.
\end{lemma}
\begin{proof}
We will prove the first inequality, while the second follows similarly. 
First observe that in view of Lemma \ref{lemma:variance},  
\[
\var\l( \rho'\l( \sqrt{n}\frac{(\bar \theta_j(f) - Pf) - z}{\Delta} \r) \r) \leq \var\l( \sqrt{n}\frac{(\bar \theta_j(f) - Pf) - z}{\Delta} \r)
= \frac{\sigma^2(f)}{\Delta^2},
\]
hence $\sup_{f\in \m F} \var^{1/2}\l( \rho'\l( \sqrt{n}\frac{(\bar \theta_j(f) - Pf) - z}{\Delta} \r) \r) \leq \frac{\sigma(\m F)}{\Delta}$. Next, Fact \ref{fact:ineq} (more specifically, inequality \eqref{eq:klein2}) implies that for any fixed $z\in \mb R$,
\begin{multline*}
\inf_{f\in \m F} \sum_{j=1}^k \frac{1}{\sqrt{k}} \l( \rho'\l( \sqrt{n}\frac{(\bar \theta_j(f) - Pf) - z}{\Delta} \r) - \mb E\rho'\l(\sqrt{n} \frac{(\bar \theta_j(f) - Pf) - z}{\Delta} \r)  \r) \\
\geq 
-2\mb E\sup_{f\in \m F}\frac{1}{\sqrt{k}} \sum_{j=1}^k  \l( \mb E\rho'\l(\sqrt{n} \frac{(\bar \theta_j(f) - Pf) - z}{\Delta} \r) -  \rho'\l( \sqrt{n}\frac{(\bar \theta_j(f) - Pf) - z}{\Delta} \r) \r) 
\\
- \frac{\sigma(\m F)}{\Delta}\sqrt{2s} - 8\l\|\rho' \r\|_\infty \frac{s}{3\sqrt{k}}
\end{multline*}
with probability at least $1 - e^{-s}$. 
It follows from Remark \ref{rem:sup} that $\|\rho'\|_\infty \leq 2$, hence it remains to estimate the expected supremum. 
To this end, we will apply symmetrization and Talagrand's contraction inequalities 
\footnote{We use the versions of these inequalities without absolute values inside the supremem; see \cite{medina2014learning} for the proof.}
\cite{kolt6,LT-book-1991,medina2014learning} (where functions 
$h_j(x) := -\frac{\rho'\l(x - \sqrt{n}z/\Delta\r) - \rho'\l( - \sqrt{n}z/\Delta\r)}{L(\rho')}$ are contractions satisfying $h_j(0)=0$). 
Let $\eps_1,\ldots,\eps_k$ be i.i.d. Rademacher random variables independent of $X_1,\ldots,X_N$, and note that 
\begin{multline*}
\mb E\sup_{f\in \m F}\frac{1}{\sqrt{k}} \sum_{j=1}^k \l(  -\rho'\l( \sqrt{n}\frac{(\bar \theta_j(f) - Pf) - z}{\Delta} \r) + \mb E\rho'\l(\sqrt{n} \frac{(\bar \theta_j(f) - Pf) - z}{\Delta} \r)  \r)  \\
\leq 2 L(\rho') \, \mb E\sup_{f\in \m F}\frac{1}{\sqrt{k}} \sum_{j=1}^k \eps_j  h_j\l( \sqrt{n}\frac{\bar \theta_j(f) - Pf}{\Delta} \r) 
\leq \frac{4 }{\Delta} \, \mb E\sup_{f\in \m F} \sum_{j=1}^k \eps_j \frac{\sqrt{n}}{\sqrt{k}} (\bar\theta_j(f) - Pf),
\end{multline*}
where we used the fact that $L(\rho')\leq 1$. Next, desymmetrization inequality \cite{kolt6} implies that 
\[
\mb E\sup_{f\in \m F} \sum_{j=1}^k \eps_j \frac{\sqrt{n}}{\sqrt{k}} (\bar\theta_j(f) - Pf) \leq 
\frac{2}{\sqrt{N}} \mb E\sup_{f\in \m F} \sum_{j=1}^N \l( f(X_j) - Pf \r).
\]
and the result follows. 
\end{proof}

\begin{lemma}
\label{lemma:step2}
Assume that $\mb E|f(X) - \mb Ef(X)|^{2}<\infty$ for all $f\in \m F$ and that $\rho$ satisfies Assumption \ref{ass:1}. Then for all $f\in \m F$ and $z\in \mb R$ satisfying $|z|\leq \frac{1}{2}\frac{\Delta}{\sqrt{n}}$,
\[
\l|\mb E\rho'\l(\sqrt{n} \frac{(\bar \theta_j(f) - Pf) - z}{\Delta} \r) - 
\mb E \rho'\l( \frac{W(f) - \sqrt{n} z }{\Delta} \r) \r|  \leq 
2 \,G_f(n,\Delta). 
\]
\end{lemma}
\begin{proof}
Let $T(x) = x - \rho'(x)$, and note that $T(x)=0$ for $|x| \leq 1$ by Assumption \ref{ass:1}. Moreover, $T$ is non-decreasing. As 
\[
\mb E\l(\sqrt{n} \frac{(\bar \theta_j(f) - Pf) - z}{\Delta} - \frac{W(f) - \sqrt{n} z }{\Delta}\r)=0,
\] 
it is easy to check that
\begin{multline*}
\l|\mb E\rho'\l(\sqrt{n} \frac{(\bar \theta_j(f) - Pf) - z}{\Delta} \r) - 
\mb E \rho'\l( \frac{W(f) - \sqrt{n} z }{\Delta} \r) \r| 
\\
= \l| \mb E \l( T\l( \sqrt{n} \frac{(\bar \theta_j(f) - Pf) - z}{\Delta}\r) - T\l(  \frac{W(f) - \sqrt{n} z }{\Delta}  \r) \r) \r|.
\end{multline*}
For any bounded non-negative function $h:\mb R\mapsto \mb R_+$ and any signed measure $Q$, 
\begin{align*}
\l| \int_{\mb R} h(x)dQ \r| = \l| \int_{0}^{\infty} Q\l( x: \,h(x)\geq t \r) dt \r|.
\end{align*} 
Since any bounded function $h:\mb R\mapsto \mb R$ can be written as $h = h_+ - h_-$, where  $h_+=\max(h,0)$ and $h_- = \max(-h, 0)$ are both nonnegative, we deduce that 
\begin{equation}
\label{eq:c1}
\l| \int_{\mb R} h(x)dQ \r|\leq  \int_{0}^{\infty} |Q\l( x: \,h_+(x)\geq t \r)| dt +  \int_{0}^{\infty} |Q\l( x: \,h_-(x)\geq t \r)| dt. 
\end{equation} 
Moreover, if $h$ is monotone, the sets $\{x: \,h_+(x)\geq t\}$ and $\{x: \,h_-(x)\geq t\}$ are half-intervals. 
Take $h=T$; it follows from Assumption \ref{ass:1} that $T_+(x) \leq \max(x - 1,0)$, hence 
$T_+^{-1}(t)\geq 1 + t$ for $t > 0$. 
Let $\Phi_f^{(n)}(\cdot)$ stand for the cumulative distribution function of $\sqrt{n}\l( \bar\theta_j(f) - Pf \r)$. 
Applying \eqref{eq:c1} to the monotone function $x\mapsto T\l( \frac{x- \sqrt{n}z}{\Delta}\r)$ 
and $Q(\cdot)=\Phi_f^{(n)}(\cdot) - \Phi \l(\cdot/\sigma(f)\r)$, we deduce that for any $f\in \m F$,
\begin{multline}
\label{eq:decomp}
\l| \mb E \l( T\l( \sqrt{n} \frac{(\bar \theta_j(f) - Pf) - z}{\Delta}\r) - T\l(  \frac{W(f) - \sqrt{n} z }{\Delta}  \r) \r) \r| 
\\
\leq \int_0^\infty \l| Q\l( x: T_+\l( \frac{x - \sqrt{n}z}{\Delta} \r) \geq t\r) \r| dt + \int_0^\infty \l| Q\l( x: T_-\l( \frac{x - \sqrt{n}z}{\Delta} \r) \geq t\r) \r| dt.
\end{multline}
Inequality \eqref{eq:BE-main} implies that
\begin{multline}
\label{eq:key-ineq}
 \l| Q\l( x: T_+\l( \frac{x - \sqrt{n}z}{\Delta} \r) \geq t\r) \r| =  \l| Q\l( x: x \geq \Delta \, T_+^{-1}(t) + \sqrt{n}z \r) \r|
 \\
 \leq g_f\big( \Delta \, T_+^{-1}(t) + \sqrt{n}z,n\big) \leq 
 g_f\l(\Delta\l( \frac{1}{2} + t \r),n \r)
 \end{multline}
where we used monotonicity of $g_f$ together with the inequalities $\sqrt{n}|z|\leq \frac{1}{2}\Delta$ and $T_+^{-1}(t)\geq 1 + t$ on the last step. 
Integrating inequality \eqref{eq:key-ineq} from $0$ to $\infty$, we see that 
\begin{equation*}
\int_0^\infty \l| Q\l( x: T_+\l( \frac{x - \sqrt{n}z}{\Delta} \r) \geq t\r) \r| dt \leq 
\int_{0}^{\infty} g_f\l(\Delta\l( \frac{1}{2} + t \r),n \r) dt = G_f(n,\Delta).
\end{equation*}
Similarly, $\int_0^\infty \l| Q\l( x: T_-\l( \frac{x - \sqrt{n}z}{\Delta} \r) \geq t\r) \r| dt \leq G_f(n,\Delta)$, hence the conclusion follows from \eqref{eq:decomp}. 
\end{proof}

\begin{lemma}
\label{lemma:step3}
Let $\eps>0$ be such that $\eps \leq 0.045$, 
and set 
\[
z_1 = -\frac{\eps}{ 0.09}\, \frac{\max\l(\Delta,\sigma(\m F) \r)}{\sqrt{n}}. 
\]
Then for all $f\in \m F$,
\begin{equation*}
\mb E \rho'\l( \frac{W(f) - \sqrt{n} z_1 }{\Delta} \r) > \eps.
\end{equation*}
\end{lemma}
\begin{proof}
For any bounded function $h$ such that $h(-x)=-h(x)$ and $h(x)\geq 0$ for $x\geq 0$, and any $z\leq 0$,
\[
\int_\mb R h(x-z)\phi_\sigma(x) dx = \int_0^\infty h(x)\l( \phi_\sigma(x+z) - \phi_\sigma(-x+z)\r)dx \geq 0, 
\] 
where $\phi_\sigma(x) = (2\pi\sigma)^{-1/2}e^{-x^2/2\sigma^2}$. 
Recall that $H'(x) \geq \frac{x}{2}$ for $0\leq x\leq 2$, and take 
\[
h(x) := \rho'(x) - \frac{x}{2} I\{|x|\leq 2 \} = H'(x) - \frac{x}{2} I\l\{ |x|\leq 2\r\}.
\] 
Observe that $h(x)\geq 0$ for $x\geq 0$ by assumptions on $\rho$, hence for any $j$,
\begin{multline}
\label{eq:exp0}
 \mb E \rho'\l( \frac{W(f) - \sqrt{n}\, z_1 }{\Delta} \r)
  = \frac{1}{2}\mb E\l(  \frac{W(f) - \sqrt{n}\, z_1 }{\Delta} I\l\{ \l|  \frac{W(f) - \sqrt{n}\, z_1 }{\Delta} \r| \leq 2 \r\}\r) 
 + \mb E h\l(  \frac{W(f) - \sqrt{n}\, z_1 }{\Delta} \r) 
 \\
\geq \max\l( \frac{1}{2}\mb E\l(  \frac{W(f) - \sqrt{n}\, z_1 }{\Delta} I\l\{ \l|  \frac{W(f) - \sqrt{n}\, z_1 }{\Delta} \r| \leq 2 \r\}\r),
\, \mb E h\l(  \frac{W(f) - \sqrt{n}\, z_1 }{\Delta} \r) \r), 
\end{multline}
where we used the fact that both terms are nonnegative. Next, we will find lower bounds for each of the terms in the maximum above, starting with the first. 

\textbf{(1)} Consider two possibilities: (a) $\Delta < \sigma(f)$ and (b) $\Delta \geq \sigma(f)$. In the first case, we will use the trivial lower bound $\mb E\l(  \frac{W(f) - \sqrt{n}\, z_1 }{\Delta} I\l\{ \l|  \frac{W(f) - \sqrt{n}\, z_1 }{\Delta} \r| \leq 2 \r\}\r) \geq 0$. 
The main focus will be on the second case. To this end, note that $Z:=\frac{W(f)}{\sigma(f)}\sim N(0,1)$, hence
\begin{multline}
\label{eq:exp1}
\frac{1}{2}\mb E\l(  \frac{W(f) - \sqrt{n}\, z_1 }{\Delta} I\l\{ \l|  \frac{W(f) - \sqrt{n}\, z_1 }{\Delta} \r| \leq 2\r\}\r) 
\\
= \frac{\sigma(f)}{2\Delta}\mb E \l( Z \,I\l\{ \l| Z - \frac{\sqrt{n}\, z_1}{\sigma(f)} \r|\leq 2\frac{\Delta}{\sigma(f)} \r\} \r) - 
\frac{\sqrt{n}\, z_1}{2\Delta}\pr{\l| Z - \frac{\sqrt{n}\, z_1}{\sigma(f)} \r|\leq 2\frac{\Delta}{\sigma(f)}}.
\end{multline}
Direct computation shows that for any $a\in \mb R, \ t>0$,
\begin{equation}
\label{eq:exp}
\Big| \mb E \l( Z \, I\l\{ |Z-a|\leq t\r\}\r) \Big| = \frac{1}{\sqrt{2\pi}}e^{-\frac{a^2+t^2}{2}} \l| e^{at} - e^{-at} \r|.
\end{equation}
Take $a= - \frac{z_1\sqrt{n}}{\sigma(f)}$, $t=2\frac{\Delta}{\sigma(f)}$, and observe that assumptions of the Theorem imply the inequality $|a| \leq \frac{t}{4}$. 
Minimum of the function $a\mapsto a^2 + t^2 - 2|a|t$ over the set $0\leq a\leq t/4$ is attained at $a=t/4$, implying that  
$a^2 + t^2 - 2|a| t \geq \frac{9}{16}t^2>\frac{t^2}{2}$.  
Combining this with \eqref{eq:exp}, we deduce that 
\begin{equation*}
\Big| \mb E \l( Z \, I\l\{ |Z-a|\leq t\r\}\r) \Big| \leq  \frac{1}{\sqrt{2\pi}}e^{-t^2/4} e^{-|at|}\l| e^{at} - e^{-at}\r|
= \frac{e^{-t^2/4}}{\sqrt{2\pi}}\l( 1 - e^{-2|at|}\r)\leq \frac{e^{-t^2/4}}{\sqrt{2\pi}}\cdot 2|at|,
\end{equation*}
hence
\begin{equation*}
\Bigg| \frac{\sigma(f)}{2\Delta}\mb E \l( Z \,I\l\{ \l| Z - \frac{\sqrt{n}\, z_1}{\sigma(f)} \r|\leq 2\frac{\Delta}{\sigma(f)} \r\} \r) \Bigg| \leq \frac{2}{\sqrt{2\pi}} \l|\frac{z_1\sqrt{n}}{\sigma(f)}\r| e^{-\frac{\Delta^2}{\sigma^2(f)}}
= \frac{2}{\sqrt{2\pi}} \l|\frac{z_1\sqrt{n}}{\Delta}\r| \frac{\Delta}{\sigma(f)}e^{-\frac{\Delta^2}{\sigma^2(f)}}.
\end{equation*}
Moreover, since $|z_1|\leq \frac{1}{2}\frac{\Delta}{\sqrt{n}}$ by assumptions of the lemma, it follows that 
\[
\pr{\l| Z - \frac{\sqrt{n}\, z_1}{\sigma(f)} \r|\leq 2\frac{\Delta}{\sigma(f)}} \geq
\pr{\l| Z\r|\leq \frac{3\Delta}{2\sigma(f)} }\geq 1 - 2\Phi(-3/2)>0.86.
\]
Together with \eqref{eq:exp0}, \eqref{eq:exp1}, the last display yields that 
\[
 \mb E \rho'\l( \frac{W(f) - \sqrt{n}\, z_1 }{\Delta} \r) > \l|\frac{0.86}{2} \frac{z_1\sqrt{n}}{\Delta} \r| - 
 \frac{2}{\sqrt{2\pi}}\l|\frac{z_1\sqrt{n}}{\Delta}\r| \frac{\Delta}{\sigma(f)} e^{-\frac{\Delta^2}{\sigma^2(f)}}.
\]
As $x\mapsto x e^{-x^2}$ is decreasing for $x\geq 1/\sqrt{2}$, one easily checks that $\frac{\Delta}{\sigma(f)} e^{-\frac{\Delta^2}{\sigma^2(f)}}\leq e^{-1}$ as $\Delta \geq \sigma(f)$, hence  
\begin{equation*}
\mb E \rho'\l( \frac{W(f) - \sqrt{n}\, z_1 }{\Delta} \r) > \l( 0.43 -  \frac{2}{e\sqrt{2\pi}} \r)|z_1| \frac{\sqrt{n}}{\Delta} > 0.1364 |z_1| \frac{\sqrt{n}}{\Delta}.
\end{equation*}

\textbf{(2)} For the second term, we start with a simple inequality 
\begin{multline*}
\mb E h\l(  \frac{W(f) - \sqrt{n}\, z_1 }{\Delta} \r) \geq \mb E \rho'\l(  \frac{W(f) - \sqrt{n}\, z_1 }{\Delta} \r) \, 
I\l\{ \l|  \frac{W(f) - \sqrt{n}\, z_1 }{\Delta} \r| > 2 \r\} 
\\
\geq \underbrace{\rho'(2)}_{\geq 1} \, \mb E\l( I\l\{ \frac{W(f) - \sqrt{n}\, z_1 }{\Delta} > 2 \r\} - I\l\{ \frac{W(f) - \sqrt{n}\, z_1 }{\Delta} < -2 \r\} \r)
\end{multline*}
which follows from the definition of $h$ and assumptions on $\rho$. 
Again, we consider two possibilities: (a) $\Delta < \sigma(f)$ and (b) $\Delta \geq \sigma(f)$. In case (b), we use the trivial bound 
\[
\mb E\l( I\l\{ \frac{W(f) - \sqrt{n}\, z_1 }{\Delta} > 2 \r\} - I\l\{ \frac{W(f) - \sqrt{n}\, z_1 }{\Delta} < -2 \r\} \r) \geq 0.
\] 
In the first case, we see that 
\begin{multline*}
\Pr\l( \frac{W(f) - \sqrt{n}\, z_1 }{\Delta} \geq 2 \r) - \Pr\l( \frac{W(f) - \sqrt{n}\, z_1 }{\Delta} \leq -2  \r) 
\\
= \Pr\l( Z \geq \frac{\sqrt{n}z_1}{\sigma(f)} + 2\frac{\Delta}{\sigma(f)} \r) - 
\Pr\l( Z \leq \frac{\sqrt{n}z_1}{\sigma(f)} - 2\frac{\Delta}{\sigma(f)}\r)
\\ 
= \Pr\l( Z \in \l[ \frac{\sqrt{n}\,z_1}{\sigma(f)} + 2\frac{\Delta}{\sigma(f)}, -\frac{\sqrt{n}\, z_1}{\sigma(f)} + 2\frac{\Delta}{\sigma(f)} \r]\r).
\end{multline*}
Lemma \ref{lemma:normal} implies that 
\begin{multline*}
\Pr\l( Z \in \l[ \frac{\sqrt{n}\, z_1}{\sigma(f)} + 2\frac{\Delta }{\sigma(f)}, -\frac{\sqrt{n}\, z_1}{\sigma(f)} + 2\frac{\Delta}{\sigma(f)} \r]\r) \geq 2 e^{-\frac{2\Delta^2}{\sigma^2(f)}} \Pr\l( Z\in \l[0, \frac{\sqrt{n}|z_1|}{\sigma(f)}\r] \r)
\\
\geq 2 e^{-2}\Pr\l( Z\in \l[0, \frac{\sqrt{n}|z_1|}{\sigma(f)}\r] \r),
\end{multline*}
where we used the fact that $\Delta < \sigma(f)$ by assumption. 
Finally, Lemma \ref{lemma:quantile} implies that 
\[
\Pr\l( Z\in \l[0, \frac{\sqrt{n}|z_1|}{\sigma(f)}\r] \r) > \frac{1}{3}\frac{\sqrt{n}|z_1|}{\sigma(f)}
\]
whenever $|z_1|\leq 0.99 \frac{\sigma(f)}{\sqrt{n}}$. 
In conclusion, we demonstrated that in case (a) 
\[
\mb E h\l(  \frac{W(f) - \sqrt{n}\, z_1 }{\Delta} \r) > \frac{2e^{-2}}{3} |z_1|\frac{\sqrt{n}}{\sigma(f)} > 0.09 |z_1|\frac{\sqrt{n}}{\sigma(f)}. 
\]
Combining results \textbf{(1)} and \textbf{(2)} for both terms in the maximum \eqref{eq:exp0}, we see that for any $\Delta > 0$, 
\begin{equation}
\label{eq:finalBound}
\mb E \rho'\l( \frac{W(f) - \sqrt{n}\, z_1 }{\Delta} \r) > \min\l( 0.1364,0.09\r) |z_1| \frac{\sqrt{n}}{\max(\Delta,\sigma(f))}
= 0.09 |z_1| \frac{\sqrt{n}}{\max(\Delta,\sigma(f))}
\end{equation}
given that $|z_1|\leq \frac{1}{2} \frac{\max\l(\Delta ,\sigma(f)\r)}{\sqrt{n}}$. 
Let $\eps>0$. It is easy to check that setting 
\begin{equation*}
z_1 = -\frac{1}{0.09} \max\l(\Delta,\sigma(f)\r) \frac{\eps}{\sqrt{n}}
\end{equation*}
yields, in view of \eqref{eq:finalBound}, that
\[
\mb E \rho'\l( \frac{W(f) - \sqrt{n} z }{\Delta} \r) > \eps,
\]
as long as condition $|z_1|\leq \frac{1}{2}\frac{\max\l(\Delta,\sigma(f)\r)}{\sqrt{n}}$ holds for all $j$. The latter is equivalent to requirement that 
$\eps \leq \frac{0.09}{2}$. 
\end{proof}

\subsection{Proof of Theorem \ref{th:unif-U}.}
\label{proof:unif-U}

Suppose $z_1,z_2$ are such that on an event of probability close to $1$, $U'_{N,n}(z_1;f) > 0$ and $U'_{N,n}(z_2;f) < 0$ for all $f\in \m F$ simultaneously. 
For such $z_1,z_2$, it is easy to see that $\widetilde \theta^{(k)}(f) - Pf \in (z_1,z_2)$ for all $f\in \m F$ on the corresponding event. 
Observe that 
\begin{multline*}
U'_{N,n}(z;f) = \frac{1}{N!}\sum_{(i_1,\ldots,i_N)\in \pi_N} T_{i_1,\ldots,i_N}(z;f) - \mb E \, T_{i_1,\ldots,i_N}(z;f) 
\\ 
+ \mb E  \rho'\l(\sqrt{n}\,\frac{\bar \theta(f;\{1,\ldots,n\}) - Pf - z}{\Delta}\r) 
- \mb E \rho'\l( \frac{W(f) - \sqrt{n} z }{\Delta} \r)  + 
\mb E \rho'\l( \frac{W(f) - \sqrt{n} z }{\Delta} \r).
\end{multline*}
The rest of the proof mimics the steps in the proof of Theorem \ref{th:unif}. 
Namely, we will find positive $\eps_1, \eps_2$ such that 
\begin{equation}
\label{eq:eps-1.2}
\inf_{f\in \m F} \frac{1}{N!}\sum_{(i_1,\ldots,i_N)\in \pi_N} T_{i_1,\ldots,i_N}(z;f) - \mb E \, T_{i_1,\ldots,i_N}(z;f)  
=  \inf_{f\in \m F} \l( U'_{N,n}(z;f) - \mb E U'_{N,n}(z;f) \r) \geq -\eps_1
\end{equation}
with high probability and 
\[
\inf_{f\in \m F} \mb E  \rho'\l(\sqrt{n}\,\frac{\bar \theta(f;\{1,\ldots,n\}) - Pf - z}{\Delta}\r) 
- \mb E \rho'\l( \frac{W(f) - \sqrt{n} z }{\Delta} \r) \geq -\eps_2
\] 
and will choose $z$ such that $\mb E \rho'\l( \frac{W(f) - \sqrt{n} z }{\Delta} \r)>\eps_1+\eps_2$. 

In view of \eqref{eq:U-concentr} and the fact that $\|\rho'\|_\infty\leq 2$, $\eps_1$ can be chosen as 
\[
\eps_1 = 2\mb E\sup_{f\in \m F}\Big( T_{1,\ldots,N}(z;f) - \mb E T_{1,\ldots,N}(z;f) \Big)
+ \frac{\sigma(f)}{\Delta}\sqrt{\frac{2s}{k}} + \frac{16s}{3k}
\]
for which inequality \eqref{eq:eps-1.2} holds with probability at least $1-e^{-s}$. 
It is also easy to see, following the symmetrization-contraction argument of Lemma \ref{lemma:step1}, that 
\[
\mb E \sup_{f\in \m F}\Big( T_{1,\ldots,N}(z;f) -\mb E T_{1,\ldots,N}(z;f)\Big) \leq
\frac{8}{\Delta \sqrt{k}} \, \mb E\sup_{f\in \m F} \l|\frac{1}{\sqrt{nk}} \sum_{j=1}^{nk} \l( f(X_j) - Pf \r)\r|. 
\] 
Lemma \ref{lemma:step2} implies that 
\[
\sup_{f\in \m F} \l| \mb E  \rho'\l(\sqrt{n}\,\frac{\bar \theta(f;\{1,\ldots,n\}) - Pf - z}{\Delta}\r) 
- \mb E \rho'\l( \frac{W(f) - \sqrt{n} z }{\Delta} \r) \r| \leq 2\sup_{f\in\m F} G_f(n,\Delta):= \eps_2.
\]
Finally, equation \eqref{eq:finalBound} in the proof of Lemma \ref{lemma:step3} yields that, for $z<0$,  
\[
\mb E \rho'\l( \frac{W(f) - \sqrt{n}\, z }{\Delta} \r) > 0.09 |z| \frac{\sqrt{n}}{\max(\Delta,\sigma(f))}
\]
as long as $|z|\leq \frac{1}{2}\frac{\max\l(\Delta,\sigma(f)\r)}{\sqrt{n}}$. 
Hence, choosing $z_1:= -\frac{1}{0.09}\frac{\max\l(\Delta,\sigma(\m F)\r)}{\sqrt{n}}\l( \eps_1 + \eps_2\r)$ implies that $U'_{N,n}(z_1;f)>0$ for all $f\in \m F$ simultaneously with probability at least $1-e^{-s}$. Similarly, $z_2=-z_1$ satisfies $U'_{N,n}(z_2;f)<0$ for all $f\in \m F$ with the same probability, and the claim follows.

\subsection{Proof of Theorem \ref{th:unif2}.}
\label{proof:unif2}

The proof closely follows the steps of the proof of Theorem \ref{th:unif}. 
All the probabilities below are evaluated conditionally on $N_J$ (see section \ref{sec:adversary} for the definition). 
Recall that  
\[
G_k(z;f) =  \frac{1}{\sqrt k} \sum_{j=1}^k \rho'\l( \sqrt{n}\,\frac{ (\bar \theta_j(f) - Pf) - z}{\Delta} \r).
\]
We are looking for $z_1,z_2\in \mb R$ with $|z_1|,\, |z_2|$ as small as possible such that on an event of probability close to $1$, $G_k(z_1;f) > 0$ and $G_k(z_2;f) < 0$ for all $f\in \m F$ simultaneously. 
Observe that 
\[
G_k(z;f) = \frac{1}{\sqrt k} \sum_{j\in J} \rho'\l( \sqrt{n}\,\frac{ (\bar \theta_j(f) - Pf) - z}{\Delta} \r) + 
\frac{1}{\sqrt k} \sum_{j\notin J}\rho'\l( \sqrt{n}\,\frac{ (\bar \theta_j(f) - Pf) - z}{\Delta} \r).
\]
The second sum can be estimated as 
\[
\l| \frac{1}{\sqrt k} \sum_{j\notin J}  \rho'\l( \sqrt{n}\,\frac{ (\bar \theta_j(f) - Pf) - z}{\Delta} \r) \r|
\leq \|\rho'\|_\infty \sum_{j\notin J}\frac{1}{\sqrt k} \leq 2 \frac{\m O}{\sqrt k}.
\]
where we used the fact that $\|\rho'\|_\infty \leq 2$. 
For the first sum, we proceed as in the proof of Theorem \ref{th:unif} and decompose it as 
\begin{multline*}
\frac{1}{\sqrt k} \sum_{j\in J}  \rho'\l( \sqrt{n}\,\frac{ (\bar \theta_j(f) - Pf) - z}{\Delta} \r) 
\\
 = \frac{1}{\sqrt k} \sum_{j\in J}  \l(  \rho'\l( \sqrt{n}\frac{(\bar \theta_j(f) - Pf) - z}{\Delta} \r) - \mb E\rho'\l(\sqrt{n} \frac{(\bar \theta_j(f) - Pf) - z}{\Delta} \r)  \r)  
 \\ 
+ \frac{1}{\sqrt k} \sum_{j\in J}  \l( \mb E\rho'\l(\sqrt{n} \frac{(\bar \theta_j(f) - Pf) - z}{\Delta} \r) - 
\mb E \rho'\l( \frac{W(f) - \sqrt{n} z }{\Delta} \r) \r) 
\\
+\frac{1}{\sqrt k}  \sum_{j\in J} \mb E \rho'\l( \frac{W(f) - \sqrt{n} z }{\Delta} \r).
\end{multline*}
It follows from Lemma \ref{lemma:step1} that for any $z\in \mb R$,
\begin{multline*}
\inf_{f\in \m F} \frac{1}{\sqrt k}\sum_{j\in J} \l(  \rho'\l( \sqrt{n}\frac{(\bar \theta_j(f) - Pf) - z}{\Delta} \r) - \mb E\rho'\l(\sqrt{n} \frac{(\bar \theta_j(f) - Pf) - z}{\Delta} \r)  \r)  \\
=\sqrt{\frac{|J|}{k}} \inf_{f\in \m F} \frac{1}{\sqrt{|J|} }\sum_{j\in J}  \l(  \rho'\l( \sqrt{n}\frac{(\bar \theta_j(f) - Pf) - z}{\Delta} \r) - \mb E\rho'\l(\sqrt{n} \frac{(\bar \theta_j(f) - Pf) - z}{\Delta} \r)  \r)  
\geq -\eps_1
\end{multline*}
with probability at least $1 - e^{-s}$, where 
\[
\eps_1 = \sqrt{\frac{|J|}{k}}\l( \frac{8 }{\Delta\sqrt{N_J}} \, \mb E\sup_{f\in \m F} \sum_{j=1}^{N_J} \l( f(X_j) - Pf \r) 
+ \frac{\sigma(\m F)}{\Delta}\sqrt{2s} + \frac{16}{3}\frac{s}{\sqrt{|J|}} \r).
\]
Next, Lemma \ref{lemma:step2} implies that 
\begin{multline*}
\inf_{f\in \m F}  \sum_{j\in J} \frac{1}{\sqrt k}\l( \mb E\rho'\l(\sqrt{n} \frac{(\bar \theta_j(f) - Pf) - z}{\Delta} \r) - 
\mb E \rho'\l( \frac{W(f) - \sqrt{n} z }{\Delta} \r) \r) 
\\
= \sqrt{\frac{|J|}{k}} \inf_{f\in \m F} \sum_{j\in J} \frac{1}{\sqrt{|J|}} \l( \mb E\rho'\l(\sqrt{n} \frac{(\bar \theta_j(f) - Pf) - z}{\Delta} \r) - \mb E \rho'\l( \frac{W(f) - \sqrt{n} z }{\Delta} \r) \r)   
\leq \eps_2,
\end{multline*}
where $\eps_2 = 2\frac{|J|}{\sqrt{k}}\sup_{f\in \m F}G_f(n,\Delta)$. 
Finally we will choose $z_1<0$ such that for all $f\in \m F$,
\begin{equation*}
\frac{1}{\sqrt k}\sum_{j\in J}  \mb E \rho'\l( \frac{W(f) - \sqrt{n} z }{\Delta} \r) > \eps_1 + \eps_2 + 
2 \frac{\m O}{\sqrt k}.
\end{equation*}
Lemma \ref{lemma:step3} implies that it suffices to take 
\[
z_1 = -\frac{1}{0.09} \frac{\widetilde \Delta}{\sqrt n} \l( \eps_1 + \eps_2 + 2 \frac{\m O}{\sqrt k}\r)\frac{\sqrt k}{|J|}
\geq -C \frac{\widetilde \Delta}{\sqrt n}\frac{\eps_1 + \eps_2 + \m O/\sqrt{k} }{\sqrt k}.
\] 
To get the final form of the bound, observe that  $\mb E\sup_{f\in \m F} \sum_{j=1}^{N_J} \l( f(X_j) - Pf \r) \leq 
\mb E\sup_{f\in \m F}  \sum_{j=1}^{N} \l( f(X_j) - Pf \r)$ due to Jensen's inequality and that $N_J\geq N/2$ by assumption. 
Similarly, setting $z_2=-z_1$ guarantees that $G_k(z_2;f)<0$ for all $f\in \m F$ with probability at least $1 - e^{-s}$, hence the claim follows.

\subsection{Proof of Lemma \ref{lemma:BE-nonunif}.}
\label{proof:BE-nonunif}

For brevity, set $Y:=f(X) - Pf$, let $Y_1,\ldots,Y_n$ be i.i.d. copies of $Y$, and let $W(f)$ have normal distribution $N(0,\sigma^2(f))$. Theorem 2.2 in \cite{chen2001non} implies that for an absolute constant $C>0$ and any $t\in \mb R$, 
\begin{multline}
\label{eq:stein}
\l| \pr{\frac{1}{\sqrt{n}}\sum_{j=1}^n Y_j \geq t } - \pr{W(f) \geq t} \r| \leq g_f(t,n) 
\\
=C \l( \frac{\mb E Y^2 \, I\l\{ \frac{|Y|}{\sigma\sqrt{n}} >  1+|t/\sigma| \r\}}{\sigma^2 (1+|t/\sigma|)^2} + \frac{1}{\sqrt{n}}
\frac{\mb E |Y|^3 \, I\l\{ \frac{|Y|}{\sigma\sqrt{n}} \leq 1+|t/\sigma|\r\}}{\sigma^3 (1+|t/\sigma|)^3}\r).
\end{multline}
It is clear that $g_f(s,n)<g_f(t,n)$ for $t<s$,  $g_f(t,m)<g_f(t,n)$ for $n<m$, and that $g_f(t,n)\to 0$ as $|t|\to\infty$. To show that $g_f(t,n)$ converges to $0$ as $n\to\infty$, let $\{a_n\}_{n\geq 1}$ be any sequence such that $a_n\to\infty$, $a_n\leq \sqrt{n}$ and $a_n=o(\sqrt{n})$. Then 
\begin{multline*}
\frac{1}{\sqrt n}\mb E |Y|^3 \, I\l\{ |Y| \leq \sigma\sqrt{n}(1+|t/\sigma|)\r\} \leq 
\frac{1}{\sqrt n}\Big( \mb E|Y|^3 \, I\l\{ |Y| \leq \sigma \cdot a_n(1+|t/\sigma|)\r\} 
\\
+ \mb E|Y|^3 \, I\l\{ \sigma \cdot a_n(1+|t/\sigma|) \leq |Y| \leq \sigma\cdot\sqrt{n}(1+|t/\sigma|)\r\}\Big)
\\
\leq (\sigma + |t|)\l( \frac{a_n}{\sqrt{n}} \, \mb EY^2 + \mb E Y^2 I\l\{ |Y| \geq \sigma \cdot a_n(1+|t/\sigma|)\r\}\r),
\end{multline*}
hence 
\begin{equation*}
g_f(t,n) \leq \frac{C}{(\sigma+|t|)^2}\l( \mb E Y^2 \, I\l\{ |Y| >  a_n\l(\sigma + |t| \r)\r\} + \frac{a_n}{\sqrt n}\mb EY^2 \r)
\end{equation*}
where the latter expression converges to $0$ as $n\to\infty$.
Next, assume that $\mb E|Y|^{2+\delta}<\infty$ for some $\delta\in(0,1]$. 
Applying H\"{o}lder's inequality followed by Markov's inequality, we deduce that 
\begin{multline*}
\mb E \l(\frac{X^2}{\sigma^2} \, I\l\{ \frac{|X|}{\sigma\sqrt{n}} >  1+|t/\sigma| \r\} \r)\leq 
\l(\mb E |X/\sigma|^{2+\delta}\r)^{2/(2+\delta)} \l( \pr{\frac{|X|}{\sigma\sqrt{n}} > 1+ |t/\sigma|} \r)^{\delta/(2+\delta))}
\\
\leq \frac{\mb E |X/\sigma|^{2+\delta}}{n^{\delta/2}(1+|t/\sigma|)^{\delta}}.
\end{multline*}
Moreover, it is clear that 
\begin{equation*}
\mb E \l( \frac{|X|^3}{\sigma^3} \, I\l\{ \frac{|X|}{\sigma\sqrt{n}} \leq 1+|t/\sigma|\r\} \r) \leq \l( \sqrt{n}(1+|t/\sigma|)\r)^{1-\delta}\mb E |X/\sigma|^{2+\delta}.
\end{equation*}
Combining these inequalities with \eqref{eq:stein}, we deduce the bound \eqref{eq:g_f} for $g_f(t,n)$. 
Finally, an upper bound for $G_f(n)$ follows by integrating the inequality \eqref{eq:g_f}. 


\section{Supplementary results.}
\label{section:supplement}

\begin{lemma}
\label{lemma:quantile}
Assume that $0\leq \alpha \leq 0.33$ and let $z(\alpha)$ be such that $\Phi(z(\alpha))-1/2 = \alpha$. 
Then $z(\alpha)\leq 3\alpha$. 
\end{lemma}
\begin{proof}
It is a simple numerical fact that whenever $\alpha \leq 0.33$, $z(\alpha)\leq 1$; indeed, this follows as $\Phi(1)\simeq 0.8413>1/2+0.33$. 
Since $e^{-y^2/2}\geq 1 - \frac{y^2}{2}$, we have 
\begin{equation}
\label{eq:05}
\sqrt{2\pi} \alpha = \int_0^{z(\alpha) } e^{-y^2/2} dy 
\geq z(\alpha) - \frac{1}{6}\l( z(\alpha) \r)^3
\geq \frac{5}{6} z(\alpha),
\end{equation} 
Equation \eqref{eq:05} implies that 
$z(\alpha) \leq \frac{6}{5}\sqrt{2\pi} \, \alpha$. 
Proceeding again as in \eqref{eq:05}, we see that 
\begin{equation*}
\sqrt{2\pi} \alpha \geq 
z(\alpha) - \frac{1}{6}\l( z(\alpha) \r)^3
\\
\geq 
z(\alpha) - \frac{12\pi}{25}\alpha^2 z(\alpha)
\\
\geq z(\alpha) \l( 1 - 1.51 \, \alpha^2 \r),
\end{equation*}
hence $z(\alpha) \leq \frac{\sqrt{2\pi}}{1 - 1.51 \, \alpha^2}\, \alpha.$ 
The claim follows since $\alpha \leq 0.33$ by assumption, and $ \frac{\sqrt{2\pi}}{1 - 1.51 \cdot 0.33^2}<3$. 
\end{proof}

The following fact is well known; we present a short proof for reader's convenience. 
\begin{lemma}
\label{lemma:normal}
Let $\m A\subset \mb R$ be symmetric, meaning that $\m A=-\m A$, and let $Z\sim N(0,1)$. Then for all $x\in \mb R$,
\[
\pr{Z\in A - x}\geq e^{-x^2/2}\,\pr{Z\in A}.
\]
\end{lemma}
\begin{proof}
Observe that
\begin{align*}
\pr{Z\in A} & = \int_\mb R I\{z\in A\} \frac{1}{\sqrt{2\pi}}e^{-z^2/2} dz = e^{x^2/2}\int_\mb R I\{z\in A\} e^{-xz/2}e^{xz/2}\frac{1}{\sqrt{2\pi}}e^{-z^2/2} e^{-x^2/2} dz 
\\
& \leq e^{x^2/2} \sqrt{ \int_\mb R I\{z\in A\} \frac{1}{\sqrt{2\pi}}e^{-(z-x)^2/2}dz}\sqrt{ \int_\mb R I\{z\in A\} \frac{1}{\sqrt{2\pi}}e^{-(z+x)^2/2}dz} 
\\
& =  e^{x^2/2} \int_\mb R I\{z\in A\} \frac{1}{\sqrt{2\pi}}e^{-(z-x)^2/2}dz = e^{x^2/2}\,\pr{Z\in A - x},
\end{align*}
and the claim follows.
\end{proof}

\begin{lemma}
\label{lemma:variance}
Let $\rho$ satisfy Assumption \ref{ass:1}. Then for any random variable $Y$ with $\mb EY^2<\infty$, 
\[
\var\l(\rho'(Y)\r) \leq \var\l( Y \r).
\]
\end{lemma}
\begin{proof}
The function $u(t)=t-\rho'(t) - c$ is nondecreasing for any $c\in \mb R$ by Assumption \ref{ass:1}, hence for any $s\in \mb R$, there exists $t(s)$ such that $u(t)\geq s$ iff $t\geq t(s)$. Next, consider $\tilde u(t) = t - \mb EY - (\rho'(t) - \mb E\rho'(Y))$, and let $t_0$ be such that $\tilde u(t)\geq 0$ for $t\geq t_0$ and $\tilde u(t) \leq 0$ when $t \leq t_0$. 
For any $x,y\in \mb R$, $x^2 - y^2\geq 2y(x-y)$. Letting $x=t - \mb EY$ and $y = \rho'(t) - \mb E\rho'(Y)$, we obtain 
\begin{equation*}
(t - \mb EY)^2 - (\rho'(t) - \mb E\rho'(Y))^2 \geq 2(\rho'(t) - \mb E\rho'(Y))\l(t - \mb EY - (\rho'(t) - \mb E\rho'(Y))\r). 
\end{equation*}
Consider two cases: (a) $t\geq t_0$ and (b) $t<t_0$. In the first case, $\rho'(t) - \mb E\rho'(Y) \geq \rho'(t_0) - \mb E \rho'(Y)$ and $\l(t - \mb EY - (\rho'(t) - \mb E\rho'(Y))\r)\geq 0$, hence 
\[
(t - \mb EY)^2 - (\rho'(t) - \mb E\rho'(Y))^2 \geq 2(\rho'(t_0) - \mb E\rho'(Y))\l(t - \mb EY - (\rho'(t) - \mb E\rho'(Y))\r).
\]
In the second case, $\rho'(t) - \mb E\rho'(Y) \leq \rho'(t_0) - \mb E \rho'(Y)$ and $\l(t - \mb EY - (\rho'(t) - \mb E\rho'(Y))\r)\leq 0$, hence again 
\begin{equation}
\label{eq:convex}
(t - \mb EY)^2 - (\rho'(t) - \mb E\rho'(Y))^2 \geq 2(\rho'(t_0) - \mb E\rho'(Y))\l(t - \mb EY - (\rho'(t) - \mb E\rho'(Y))\r).
\end{equation}
Replacing $t$ by $Y$ in \eqref{eq:convex} and taking the expectation yields the result. 
\end{proof}

\begin{lemma}
\label{lemma:minmax}
Let $p\in[2,3]$, and assume that $X\in \mb R^d$ has distribution $P$ from a class $\m P$ of all distributions with bounded $p$-th moments of one-dimensional projections, meaning that $\sup_{\|v\|_2=1}\mb E_P\l| \dotp{X-\mb EX}{v}\r|^p \leq 1$. 
Let $X_1,\ldots,X_N$ be a sample from $(1-\eps)P + \eps Q$ where $0\leq \eps<1/2$, $P\in \m P$ and $Q$ is an arbitrary distribution. 
Then for any estimator $\widehat \mu(X_1,\ldots,X_N)$ of the mean $\mu(P)$, 
\[
\sup_{P\in \m P,Q} \pr{\|\widehat \mu - \mu\|_2 \geq c_1\l( \sqrt{\frac{\tr(\Sigma)}{N}} \vee \eps^{1-1/p} \r)} \geq c_2.
\] 
where $\Sigma$ is the covariance matrix of $P$ and $c_1,c_2$ are absolute constants. 
\end{lemma}
\begin{proof}
If $\eps=0$, then it is well known that the minimax rate of estimating the mean is $\sqrt{\frac{\tr \Sigma}{N}}$ (e.g. see the remark following Theorem 3 in \cite{lugosi2018near}). 
Let's assume that $\eps>0$ and suppose that $P_1,P_2$ are two distributions supported on $\l\{ \eps^{-1/p},0,-\eps^{1/p}\r\}$, namely, 
$P_1(0) = P_2(0) = \frac{1-2\eps}{1-\eps},$ $P_1\l(\eps^{-1/p}\r)=\frac{\eps}{1-\eps}$, $P_2\l( -\eps^{-1/p}\r) = 0$, 
$P_2\l( \eps^{-1/p}\r)=0$, $P_2\l( -\eps^{-1/p} \r)=\frac{\eps}{1-\eps}$. 
Clealy, $P_1\in \m P, \ P_2\in \m P$, the means of $P_1, P_2$ are $\mu(P_1)=\frac{\eps^{1-1/p}}{1-\eps}$, $\mu(P_2)=-\frac{\eps^{1-1/p}}{1-\eps}$ respectively and $\l|\mu(P_1)-\mu(P_2)\r| = \frac{2\eps^{1-1/p}}{1-\eps}$.
 
Next, let $Q_1$ and $Q_2$ be Dirac measures, namely $Q_1=\delta_{\eps^{-1/p}}$ and $ Q_2 = \delta_{-\eps^{-1/p}}$. 
Then it is easy to check that $(1-\eps)P_1 + \eps Q_2 = (1-\eps) P_2 + \eps Q_1 = \hat P$, hence, given samples from $\hat P$, it is impossible to distinguish between $P_1$ and $P_2$. 
Conclusion now follows from Theorem 5.1 in \cite{chen2018robust}. 
\end{proof}

\section*{Acknowledgements. }

Part of this work was completed during my visit to ENSAE ParisTech. I would like to thank Faculty members and students at ENSAE for their hospitality.

\bibliographystyle{imsart-nameyear}	
\bibliography{RobustERM}

\end{document}